\documentclass[11pt,reqno,twoside]{amsart}

\usepackage{amsfonts}
\usepackage{amsmath,amssymb}
\usepackage{epic}
\usepackage{pstricks}
\usepackage{graphicx}

\usepackage{amsfonts,amsmath,amssymb}
\usepackage{mathrsfs,mathtools}
\usepackage{enumerate}
\usepackage{hyperref}
\usepackage{esint}
\usepackage{graphicx}
\usepackage{bm}
\usepackage{commath}
\usepackage{esint}
\DeclareMathAlphabet{\mathpzc}{OT1}{pzc}{m}{it}

\usepackage{color}

\usepackage[textsize=small]{todonotes}

\newtheorem{theorem}{Theorem}[section]
\newtheorem{proposition}[theorem]{Proposition}
\newtheorem{lemma}[theorem]{Lemma}
\newtheorem{definition}[theorem]{Definition}
\newtheorem{example}[theorem]{Example}
\newtheorem{remark}[theorem]{Remark}
\newtheorem{assumption}[theorem]{Assumption}

\numberwithin{equation}{section}
\numberwithin{theorem}{section}

\def\R{{\mathbb R}}

\def\cJ{{\mathcal J}}

\DeclareMathAlphabet\gothic{U}{euf}{m}{n}

\usepackage[textsize=small]{todonotes}
\def\cV{{\mathcal V}}
\def\cW{{\mathcal W}}

\def\RR{\mathbb{R}}

\def\NN{\mathbb{N}}

\def\Om{\Omega}

\def\bOm{\overline{\Om}}

\def\pOm{\partial\Omega}
\def\sgn{\operatorname{sgn}}

\newcommand{\Zad}{Z_{\text ad}}

\usepackage[dvips,bottom=1.4in,right=1in,top=1in, left=1in]{geometry}

\makeatletter
\def\eqnarray{\stepcounter{equation}\let\@currentlabel=\theequation
\global\@eqnswtrue
\tabskip\@centering\let\\=\@eqncr
$$\halign to \displaywidth\bgroup\hfil\global\@eqcnt\z@
  $\displaystyle\tabskip\z@{##}$&\global\@eqcnt\@ne
  \hfil$\displaystyle{{}##{}}$\hfil
  &\global\@eqcnt\tw@ $\displaystyle{##}$\hfil
  \tabskip\@centering&\llap{##}\tabskip\z@\cr}

\def\endeqnarray{\@@eqncr\egroup
      \global\advance\c@equation\m@ne$$\global\@ignoretrue}

\title{Optimal control of fractional semilinear PDEs}

\author{Harbir Antil}
\address{Department of Mathematical Sciences, George Mason University, Fairfax, VA 22030, USA.}
\email{hantil@gmu.edu}

\author{Mahamadi Warma}
\address{University of Puerto Rico  (Rio Piedras Campus), College of Natural Sciences,
Department of Mathematics, PO Box 70377 San Juan PR
00936-8377 (USA). }
\email{mahamadi.warma1@upr.edu, mjwarma@gmail.com}

\thanks{The work of the first author is partially supported by NSF grants DMS-1521590 and DMS-1818772 and Air Force Office of Scientific Research under Award NO: FA9550-19-1-0036.  The work of the second author is partially supported by the Air Force Office of Scientific Research under Award NO: FA9550-18-1-0242.}



\begin{document}

\begin{abstract}
In this paper we consider the optimal control of semilinear fractional PDEs with both spectral and 
integral fractional diffusion operators of order $2s$ with $s \in (0,1)$. We first prove the boundedness 
of solutions to both semilinear fractional PDEs under minimal regularity assumptions on domain
and data. We next introduce an 
optimal growth condition on the nonlinearity to show the Lipschitz continuity of the solution map for 
the semilinear elliptic equations with respect to the data. 
We further apply our ideas to show existence of solutions to optimal control problems with semilinear fractional equations as constraints. Under the standard assumptions on the nonlinearity (twice continuously differentiable) we derive the first and second order optimality conditions.
\end{abstract}

\maketitle

%%%%%%%%%%%%%%%%%%%%%%%%%%%%%%%%%%%%%%%%%%%%%%%%%%
\section{Introduction}
%%%%%%%%%%%%%%%%%%%%%%%%%%%%%%%%%%%%%%%%%%%%%%%%%%

Fractional order operators have recently emerged as a modeling alternative in various branches of
science. Their success can be attributed to the fact that fractional order operators can capture sharp transitions across an interface. 
These are nonlocal operators that can model multiscale behavior. One such instance 
occurs in electrical signal propagation in cardiac tissue where the presence of 
fractional diffusion has been experimentally validated \cite{AOBueno_DKay_VGrau_BRodriquez_KBurrage_2014a}. There are several applications 
where the underlying models are semilinear and contain fractional diffusion operators: 
phase field models \cite{HAntil_SBartels_2017a}, fluid dynamics \cite{csm,MR1252829}, 
diffusion of biological species \cite{viswanathan1996levy} (see \cite{cantrell2007advection} for the local case). For completeness we also mention the role of fractional operators in imaging science \cite{HAntil_SBartels_2017a, HAntil_CNRautenberg_2018b}.
Even with the existing applications, the emphasis of growing literature on fractional diffusion equations has been 
largely limited to linear fractional PDEs. 
The purpose of this paper is to consider two prototypical semilinear fractional PDEs and the corresponding
optimal control problems.

% Problem
Let $\Omega\subset\RR^N$ be an arbitrary bounded open set with boundary $\pOm$. In this paper we  investigate the well-posedness of the semilinear fractional equation
\begin{equation}\label{ellip-pro}
(L_D)^su +f(x,u)=z\quad \mbox{ in }\;\Omega,\\
\end{equation} 
and we also consider an optimal control problem 
\begin{equation}\label{cost}
\min_{z\in \Zad} J(u,z) := J_1(u) + J_2(z),
%:= \frac12 \left(\|u-u_d\|^2_{L^2(\Omega)}+\mu \|z\|_{L^2(\Omega)}^2\right)
\end{equation}
subject to the state equation \eqref{ellip-pro} and the control constraints 
\begin{equation}\label{ctrl}
z\in \Zad :=\Big\{z\in L^\infty(\Omega) \;:\; z_a \le z \le z_b , \mbox{ a.e. in } \Omega \Big\} . 
\end{equation}
Here $z_a, z_b \in L^\infty(\Omega)$ with $z_a(x) \le z_b(x)$ for a.e. $x\in\Omega$. The precise conditions on $J_1$ and $J_2$ will be given in Section~\ref{s:red_prob}  and 
Remark~\ref{rem:non_smooth}.

In \eqref{ellip-pro}, $f:\Omega\times\RR\to\RR$ is measurable and satisfies certain conditions (that we shall specify later) and $(L_D)^s$ ($0<s<1$) denotes the fractional $s$ powers of the realization in $L^2(\Om)$ of the operator $L$ formally given by
\begin{align}\label{OPL}
Lu:=-\sum_{i,j=1}^ND_i\Big(a_{ij}(x)D_j u\Big),\;\;D_i:=\frac{\partial}{\partial x_i},
\end{align}
with the Dirichlet boundary condition $u=0$ on $\pOm$.  The coefficients $a_{ij}$ are assumed to be measurable, belong to $L^\infty(\Omega)$, are symmetric, that is,
\begin{align*}
a_{ij}(x)=a_{ji}(x)\;\;\forall\;i,j=1,\cdots,N \mbox{ and for  a.e. }  x\in\Omega,
\end{align*}
and satisfy the ellipticity condition, that is, there exists a constant $\gamma>0$ such that
\begin{align*}
\sum_{i,j=1}^Na_{ij}(x)\xi_i\xi_j\ge \gamma|\xi|^2,\;\;\forall\;\xi\in \RR^N\;\mbox{ and a.e. } x\in\Omega.
\end{align*}

Besides Equation \eqref{ellip-pro}, we also consider the following semilinear elliptic system
\begin{equation}\label{ellip-int-Frac}
\begin{cases}
(-\Delta)^su+f(x,u)=z\quad&\mbox{ in }\;\Omega,\\
u=0&\mbox{ in }\;\RR^N\setminus\Omega,
\end{cases}
\end{equation} 
where $(-\Delta)^s$ denotes the integral fractional Laplace operator (see Section \ref{sec-int-frac}), together with the optimal control problem \eqref{cost} and the control constraints \eqref{ctrl}.

Notice that both $(L_D)^s$ and $(-\Delta)^s$ are nonlocal operators if $0<s<1$ and $f$ is nonlinear with respect to $u$.  This makes it  challenging to identify the minimal assumptions on $\Omega$, $f$ and $z$ in the study of the existence, uniqueness, regularity and the numerical analysis of \eqref{ellip-pro} and \eqref{ellip-int-Frac}. The main contributions of this paper are summarized as follows:
\begin{enumerate}[$(i)$]
 \item We identify minimal conditions on $f$ without any regularity assumption on the domain $\Omega$ that lead to existence,
  uniqueness and boundedness of solutions to \eqref{ellip-pro} and \eqref{ellip-int-Frac}. Our main assumption reads that $f$ is monotone in the second variable and $f(x,t)\to\infty$ as $t\to\infty$.
  
 \item We introduce an optimal growth condition on $f$ (see \eqref{growth}) that allows us to prove the Lipschitz
  continuity of the solution map: $z\mapsto S(z) := u$. Usually, a local Lipschitz continuity on $f$ 
  is assumed in most of the literature. In absence of this Lipschitz continuity 
  we also prove existence of solution to \eqref{cost}. Our growth condition is not a regularity assumption on $f$ and therefore is weaker than a local Lipschitz continuity.

 \item We study the optimality conditions for the optimization problems and under 
  standard assumptions on $f$ we derive the second order sufficient conditions. 
  We notice that if the pair $(N,s)$ fulfills the assumptions of Lemma~\ref{lem:nodisp}, 
  we do not get the 2-norm discrepancy.
  %\HA{This is surprising and it distinguishes 
  %the fractional control problem from its classical counterpart, i.e., $s = 1$ where one typically
  %observes the 2-norm discrepancy. I AM NOT SURE.} 
\end{enumerate}

We refer to \cite{bonnans2013perturbation,MR2583281,MR2902693,casas2012general}  
  and the references therein for the case $s = 1$ and \cite{HAntil_RKhatri_MWarma_2018a} for the exterior control of linear equation with integral fractional Laplacian. We emphasize that even though the 
  aforementioned list of references for $s=1$ is incomplete but these references provide a 
  comprehensive overview of the field of semilinear optimal control problems. 
  %of second order sufficient conditions.

To the best of our knowledge most of our results are new not only for the spectral case but
also the integral fractional case. We further notice that the results of 
$(ii)$ can be applied to the classical semilinear problems as well.
When $a_{ij}=\delta_{ij}$ where the latter denotes the Kronecker delta, we developed a complete analysis, including discretization, and error estimates, for \eqref{ellip-pro} 
in \cite{HAntil_JPfefferer_MWarma_2016a}. Such an error analysis can be directly applied to
\eqref{ellip-pro} under the usual assumptions on $\Omega$ and the coefficients $a_{ij}$.
By following the approach of \cite{HAntil_JPfefferer_MWarma_2016a} in conjunction with the
estimates for the linear problem \cite{GA:JP15} a similar error analysis can be developed for \eqref{ellip-int-Frac}.

In order to avoid repetition we will focus on the semilinear problem \eqref{ellip-pro} with the spectral fractional operator $(L_D)^s$. However, to prove our crucial results in (i) and (ii) we rely on an 
 integral representation of $(L_D)^s$ (cf.~\eqref{int-rep}). This integral representation is similar to the representation of the fractional operator $(-\Delta)^s$ (cf. Section~\ref{sec-int-frac})
and all the results discussed for $(L_D)^s$ directly transfer to $(-\Delta)^s$ under minor modifications. We refer to Section~\ref{sec-int-frac}, Remarks~\ref{new-rem}(b) and
\ref{rem:int_oc1} for more details. 

The remainder of the paper is organized as follows. In Section~\ref{preli-resul} we state some
preliminary results and introduce our function spaces. Our main work starts from 
Section~\ref{preli} where we first prove the existence of solutions to \eqref{ellip-pro}
in Sobolev spaces. We next show that the inverse of the solution operator is bounded and
continuous under the newly introduced growth condition in \eqref{growth}, we also study
the compactness of such an operator. We prove $L^\infty$-bound of solutions in Theorem~\ref{theo-bound}. We also derive an $L^\infty$-bound on the difference of two solutions $u_1,u_2$
corresponding to given $z_1,z_2$ in Proposition~\ref{pro-dif-sol} without any additional
assumptions on $f$. In Section \ref{sec-int-frac} we show that all our results also hold for the system \eqref{ellip-int-Frac} with very minor modifications in the proofs.
An example of $f$ is given in Section~\ref{s:ex}. We next prove
the existence of a solution to our optimal control problem in Section~\ref{s:red_prob} by just assuming the above mentioned growth condition
on $f$. Under additional regularity assumptions on $f$ with respect to the second variable, we derive the first order necessary and second order sufficient conditions
in Section~\ref{s:oc}. 
%\MW{Delete?
%We conclude with one numerical example in Section~\ref{s:numerics} using the spectral operator 
%$(L_D)^s$.}

%%%%%%%%%%%%%%%%%%%%%%%%%%%%%%%%%%%%%%%%%%%%%%%%%%
\section{Notation and Preliminary results}\label{preli-resul}
%%%%%%%%%%%%%%%%%%%%%%%%%%%%%%%%%%%%%%%%%%%%%%%%%%

Throughout this section without any mention, $\Omega\subset\RR^N$ ($N\ge 1$) denotes an arbitrary bounded open set with boundary $\pOm$.  For each result, if a regularity of $\Omega$ is needed, then we shall specify and if no specification is given, then we mean that the result holds without any regularity assumption on the open set.

\subsection{Fractional order Sobolev spaces}

Let $H_0^{1}(\Omega)=\overline{\mathcal D(\Omega)}^{H^{1}(\Omega)}$ where
$\mathcal D(\Omega)$ is the space of infinitely continuously differentiable functions with compact support in $\Omega$, and
\begin{align*}
H^{1}(\Omega)=\Big\{u\in L^2(\Omega):\;\int_{\Omega}|\nabla u|^2\;dx<\infty\Big\}
\end{align*}
is the first order Sobolev space endowed with the norm
\begin{align*}
\|u\|_{H^{1}(\Omega)}=\left(\int_{\Omega}|u|^2\;dx+\int_{\Omega}|\nabla u|^2\;dx\right)^{\frac 12}.
\end{align*}

Next, for $0<s<1$,  we define the fractional order Sobolev space
\begin{align*}
H^s(\Omega):=\left\{u\in L^2(\Omega):\; \int_{\Omega}\int_{\Omega}\frac{|u(x)-u(y)|^2}{|x-y|^{N+2s}}\;dxdy<\infty\right\},
\end{align*}
and we endow it with the norm defined by
\begin{align*}
\|u\|_{H^s(\Omega)}=\left(\int_{\Omega}|u|^2\;dx+\int_{\Omega}\int_{\Omega}\frac{|u(x)-u(y)|^2}{|x-y|^{N+2s}}\;dxdy\right)^{\frac 12}.
\end{align*}
We also let 
\begin{align*}
H_0^s(\Omega):=\overline{\mathcal D(\Omega)}^{H^s(\Omega)},
\end{align*}
and
\begin{align*}
H_{00}^{\frac 12}(\Omega):=\left\{u\in H^{\frac 12}(\Omega):\;\int_{\Omega}\frac{u^2(x)}{\mbox{dist}(x,\pOm)}\;dx<\infty\right\}.
\end{align*}
Note that
\begin{align}\label{norm-sob-es}
\|u\|_{H_0^s(\Omega)}=\left(\int_{\Omega}\int_{\Omega}\frac{|u(x)-u(y)|^2}{|x-y|^{N+2s}}\;dxdy\right)^{\frac 12}
\end{align}
defines a norm on $H_0^s(\Omega)$ if $\frac 12<s<1$. 
\begin{remark}\label{rem-sob}
{\em It is well-known (see e.g. \cite[Theorem 1.4.2.4 p.25]{Gris}) that if $\Omega$ has a Lipschitz continuous  boundary, then $H^s(\Omega)=H_0^s(\Omega)$ if and only if $0<s\le 1/2$. If $1/2<s<1$, then $H_0^s(\Omega)$ is a proper closed subspace of $H^s(\Omega)$. In particular, we also have that $H_{00}^{1/2}(\Omega)\subsetneqq H_0^{1/2}(\Omega)=H^{1/2}(\Omega)$. A complete description of this fact for arbitrary bounded open sets is contained in \cite{War}. }
\end{remark}
The fractional order Sobolev spaces can be also defined by using interpolation theory. That is, for every $0<s<1$,
\begin{align*}
H^s(\Omega)=[H^1(\Omega),L^2(\Omega)]_{1-s},
\end{align*}
and for every $s\in (0,1)$ we have that

\begin{align*}
H_0^s(\Omega)=[H_0^1(\Omega),L^2(\Omega)]_{1-s}\;\mbox{ if }\; s\in (0,1)\setminus\{1/2\}\;\mbox{ and }\; H_{00}^{\frac 12}(\Omega)=[H_0^1(\Omega),L^2(\Omega)]_{\frac 12}.
\end{align*}
Here, for $0<\theta<1$, $[\cdot,\cdot]_\theta$ denotes the complex interpolation space.

Since $\Omega$ is assumed to be bounded we have the following continuous embedding:
\begin{equation}\label{inj1}
H_0^s(\Omega)\hookrightarrow
\begin{cases}
L^{\frac{2N}{N-2s}}(\Omega)\;\;&\mbox{ if }\; N>2s,\\
L^p(\Omega),\;\;p\in[1,\infty)\;\;&\mbox{ if }\; N=2s,\\
C^{0,s-\frac{N}{2}}(\bOm)\;\;&\mbox{ if }\; N<2s.
\end{cases}
\end{equation}

We notice that if $N\ge 2$, then $N\ge 2>2s$ for every $0<s<1$, or if $N=1$ and $0<s<\frac 12$, then $N=1>2s$, and thus the first embedding in \eqref{inj1} will be used.  If $N=1$ and $s=\frac 12$, then we will use the second embedding. Finally, if $N=1$ and $\frac 12<s<1$, then $N=1<2s$ and hence, the last embedding will be used. 

For more details on fractional order Sobolev spaces we refer the reader to  \cite{Adam,NPV,Gris,LM,War} and their references.

\subsection{The fractional powers of the elliptic operator}\label{s:spec}

Let $L_D$ be the realization in $L^2(\Omega)$ of  $L$ given in \eqref{OPL} with the boundary condition $u=0$ on $\pOm$. That is, $L_D$ is the positive and self-adjoint operator on $L^2(\Omega)$ associated with the closed and bilinear  form

\begin{align*}
\mathcal E_D(u,v)=\sum_{i,j=1}^N\int_{\Omega}a_{ij}(x)D_i u D_j v\;dx,\;\;u,v\in H_0^{1}(\Omega),
\end{align*}
in the sense that
\begin{equation*}
\begin{cases}
D(L_D)=\left\{u\in H_0^{1}(\Omega):\;\exists\;w\in L^2(\Omega),\; \mathcal E_D(u,v)=(w,v)_{L^2(\Omega)}\;\forall\;v\in H_0^{1}(\Omega)\right\},\\
L_Du=w.
\end{cases}
\end{equation*}
It is well-known that $L_D$ has a compact resolvent and its eigenvalues form a non-decreasing sequence $0<\lambda_1\le\lambda_2\le\cdots\le\lambda_n\le\cdots$ of real numbers satisfying $\lim_{n\to\infty}\lambda_n=\infty$. We denote by $(\varphi_n)_{n\in\NN}$ the  orthonormal basis of eigenfunctions associated with $(\lambda_n)_{n\in\NN}$.

For any $\theta\geq0$, we also introduce the following fractional order Sobolev space:
\begin{align*}
\mathbb H^\theta(\Omega):=\left\{u=\sum_{n=1}^\infty u_n\varphi_n\in L^2(\Omega):\;\;\|u\|_{\mathbb H^\theta(\Omega)}^2:=\sum_{n=1}^\infty \lambda_n^\theta u_n^2<\infty\right\},
\end{align*}
where $u_n=(u,\varphi_n)_{L^2(\Omega)}=\int_{\Omega}u\varphi_n\;dx$.
If $0<s<1$, then it is well-known that
\begin{equation}\label{inf}
\mathbb H^s(\Omega)=
\begin{cases}
H_0^s(\Omega)\;\;\;&\mbox{ if }\; s\ne \frac 12,\\
H_{00}^{\frac 12}(\Omega)\;\;&\mbox{ if }\; s=\frac 12.
\end{cases}
\end{equation}
It follows from \eqref{inf} that the embedding \eqref{inj1} holds with $H_0^s(\Omega)$ replaced by $\mathbb H^s(\Omega)$.

\begin{definition}
Let $0<s<1$.
The spectral fractional $s$ powers of $L_D$  is defined on $\mathbb H^s(\Omega)$ by
\begin{align*}
(L_D)^su:=\sum_{n=1}^\infty\lambda_n^su_n\varphi_n\qquad \text{ with } \quad u_n=\int_{\Omega}u\varphi_n\;dx.
\end{align*}
\end{definition}

We notice that in this case we have
\begin{align}\label{norm-2}
\|u\|_{\mathbb H^s(\Omega)}=\|(L_D)^{\frac s2}u\|_{L^2(\Omega)}.
\end{align}
In addition $\mathcal D(\Omega)\hookrightarrow \mathbb H^s(\Omega)\hookrightarrow L^2(\Omega)$, so, the operator $(L_D)^s$ is unbounded, densely defined and with bounded inverse $(L_D)^{-s}$ in $L^2(\Omega)$. But it can also be viewed as a bounded operator from $\mathbb H^s(\Omega)$ into its dual $\mathbb H^{-s}(\Omega):=(\mathbb H^s(\Omega))^{\star}$. 
The following integral representation of $(L_D)^s$ given in \cite[Theorem 2.3]{CaSt} will be useful. For every $u,v\in\mathbb H^s(\Omega)$, we have that
\begin{align}\label{int-rep}
\langle(L_D)^su,v\rangle_{\mathbb H^{-s}(\Omega),\mathbb H^s(\Omega)}=\frac 12\int_{\Omega}\int_{\Omega}\Big(u(x)-u(y)\Big)\Big(v(x)-v(y)\Big)K_s^L(x,y)\;dxdy
+\int_{\Omega}\kappa_s(x)u(x)v(x)\;dx,
\end{align}
where

\begin{align*}
0\le K_s^L(x,y):=\frac{s}{\Gamma(1-s)}\int_0^\infty\frac{W_\Omega^L(t,x,y)}{t^{1+s}}\;dt,\;\;x,y\in\Omega,
\end{align*}
and 
\begin{align*}
0\le \kappa_s(x)=\frac{s}{\Gamma(1-s)}\int_0^\infty\left(1-e^{-tL_D}1(x)\right)\frac{dt}{t^{1+s}},\;x\in\Omega.
\end{align*}
Here, $\Gamma$ is the usual Gamma function, $(e^{-tL_D})_{t\ge 0}$ denotes the strongly continuous semigroup on $L^2(\Omega)$ generated by $-L_D$ and $W_{\Omega}^L$ is the associated heat kernel, that is,
\begin{align*}
W_\Omega^L(t,x,y)=\sum_{n=1}^\infty e^{-t\lambda_n}\varphi_n(x)\varphi_n(y),\;\;t>0,\; x,y\in\Omega.
\end{align*}

From the representation \eqref{int-rep} we immediately see that $(L_D)^s$ is a nonlocal operator. We also notice that the case of fractional powers of elliptic operators with non-zero boundary conditions has been investigated in \cite{antil2017fractional}.

For more details on fractional powers of more general operators we refer the reader to \cite{NALD,ATW,CaSt,Gru2,ST:10} and the references therein.

\subsection{Some results on Orlicz spaces}\label{s:orl}

Here we give some important properties of Orlicz type spaces that will be used throughout the paper. These results, with the exception of Remark \ref{rem-28}, are the same as the ones stated in \cite[Section 2.2]{HAntil_JPfefferer_MWarma_2016a}. We refer to \cite{HAntil_JPfefferer_MWarma_2016a} and the references therein for more details.

\begin{assumption}\label{assum2}
  For a function
  $f:\Om\times\RR\to\RR$ we consider the following assumption:
  \begin{equation*}
     \begin{cases}
        f(x,\cdot) \text{ is odd, strictly increasing}&\text{ for a.e. } x\in\Omega,\\
        f(x,0)=0 &\text{ for a.e. } x\in \Omega,\\
        f(x,\cdot) \text{ is continuous }\;&\text{ for a.e. } x\in \Omega,\\
        f(\cdot,t)  \text{ is measurable }&\mbox{ for all } t\in\RR,\\
       \displaystyle \lim_{t\to\infty}f(x,t)=\infty &\text{ for a.e. } x\in \Omega.
      \end{cases}
  \end{equation*}
 \end{assumption}
 
%Since $f(x,\cdot)$ is strictly increasing for a.e. $x\in\Omega$,
%  it has an inverse which we denote by 
  Let $\widetilde{f}(x,\cdot)$ be the inverse of $f(x,\cdot)$. Define
  $F,\widetilde{F}:\;\Omega\times\RR\to[0,\infty)$  by
  \begin{align}\label{func-F}
 F(x,t):=\int_0^{|t|}f(x,\tau)\;d\tau\;\mbox{ and  }\;
  \widetilde{F}(x,t):=\int_0^{|t|}\widetilde{f}(x,\tau)\;d\tau,\;\forall\;t\in\RR\;\mbox{ and for a.e. }x\in\Omega.
  \end{align}
  
\begin{assumption}\label{assum3}
  Under the setting of Assumption \ref{assum2}, and 
  for a.e. $x\in\Omega$, let both $F(x,\cdot)$ and $\widetilde F(x,\cdot)$ satisfy
  the global $(\triangle_2)$-condition,  
  that is, there exist two constants $c_1,c_2\in (0,1]$ independent of $x$, such that for a.e. $x\in\Omega$ and for all $t\in\RR$,
  \begin{align}\label{delta-2}
  c_1tf(x,t)\le F(x,t)\le tf(x,t)\;\mbox{ and }\; c_2t\widetilde f(x,t)\le \widetilde F(x,t)\le t\widetilde f(x,t).
  \end{align}
  \end{assumption}
  
Let
   \begin{align*}
   L_F(\Om):=\Big\{u:\Omega\to\RR\text{ measurable}: F(\cdot,u(\cdot))\in L^1(\Omega)\Big\}
  \end{align*}
  be the Musielak-Orlicz space.
  The space $L_{\widetilde F}(\Om)$ is defined similarly with $F$ replaced by $\widetilde F$.

\begin{remark}
{\em
  If Assumption \ref{assum3} holds, then by
  \cite[Theorems 1 and 2]{Doman} (see also \cite[Theorem 8.19]{Adam}),
  $L_F(\Omega)$ endowed with the Luxemburg norm given by
   \begin{align*}
\norm{u}_{F,\Omega}:=\inf\left\{k>0:\;\int_{\Omega}F \left(x,\frac{u(x)}{k}\right)\;dx\le 1\right\},
 \end{align*}
  is a reflexive Banach space. The same result also holds for $L_{\widetilde F}(\Omega)$.
 }
\end{remark}

We have the following result.

\begin{lemma}\label{lem:hoelder}{\bf\cite[Lemma 1.5]{HAntil_JPfefferer_MWarma_2016a}}
Let Assumption \ref{assum3} hold. Then  $f(\cdot,u(\cdot))\in L_{\widetilde F}(\Om)$ for all $u\in L_F(\Omega)$.
\end{lemma}

\begin{definition}
Let $0<s<1$.  Under Assumption \ref{assum3} we can define the Banach space
  $\cV_0$  by
\begin{align*}
\cV_0=\cV_0(\Omega,F):=\Big\{u\in \mathbb H^{s}(\Om): F(\cdot,u(\cdot))\in L^1(\Om)\Big\},
\end{align*}
  and we endow it with the norm defined by
 \begin{align*}
\norm{u}_{\cV_0}:=\norm{u}_{\mathbb H^{s}(\Om)}+\norm{u}_{F,\Om}.
 \end{align*}
\end{definition}

  In this case $\cV_0$ is a reflexive Banach space. 
  It follows from \eqref{inj1} that we have the continuous embedding 
  \begin{align}\label{sobo1}
  \cV_0\hookrightarrow \mathbb H^{s}(\Om)\hookrightarrow L^{2^\star}(\Omega),
  \end{align}
where we have set
\begin{align*}
2^\star=\frac{2N}{N-2s}\;\mbox{ if }\; N\ge 2>2s\;\mbox{ or if }\,N=1\;\mbox{ and }\; 0<s<\frac 12.
\end{align*}
 If $N=1$ and $s=\frac 12$, then $2^{\star}$ is any number in the interval $[1,\infty)$. If $N=1$ and $\frac 12<s<1$, then we have the continuous embedding
\begin{align}\label{sobo2}
\cV_0\hookrightarrow \mathbb H^{s}(\Om)\hookrightarrow C^{0,s-\frac 12}(\bOm).
\end{align}

We refer to \cite{Adam,Bie,RR} and their references for further properties of Orlicz type spaces.

We conclude this section with the following observation.
\begin{remark}\label{rem-28}
{\em In Assumption \ref{assum2},  the assumption that $f$ is odd can be removed. In that case, we let
\begin{align*}
\Lambda(x,t):=\int_0^tf(x,\tau)\;d\tau\;\;\mbox{ and }\;\widetilde{\Lambda}(x,t):=\int_0^t\widetilde f(x,\tau)\;d\tau,\;\;\forall\;t\in\RR\;\mbox{ and for a.e. }x\in\Omega, 
\end{align*}
and one replaces $F$ and $\widetilde F$ in \eqref{func-F} by the following: for $x\in\Omega$ fixed, we define
\begin{align}\label{New-F}
F(x,t):=\max\{\Lambda(x,t),\Lambda(x,-t)\}\;\;\mbox{ and }\; \widetilde F(x,t):=\max\{\widetilde{\Lambda}(x,t),\widetilde{\Lambda}(x,-t)\},\;\;\forall\;t\in\RR.
\end{align}
By definition, we have that $F,\widetilde F:\Omega\times\RR\to[0,\infty)$, $F(x,\cdot),\widetilde F(x,\cdot)$ are even functions (for a.e. $x\in\Omega$) as the ones in \eqref{func-F}.
Assuming that these functions satisfy Assumption \ref{assum3}, then all the results in the paper remain true without any modification in the proofs. Of course if $f(x,\cdot)$ is odd, then $F$ and $\widetilde F$ defined in \eqref{New-F} coincide with the ones given in \eqref{func-F}. We have chosen the representation \eqref{func-F} only for simplicity.}
\end{remark}

\section{Analysis of the semilinear elliptic problem}\label{preli}

In this section we give some existence, uniqueness and regularity results of weak solutions to the problems \eqref{ellip-pro} and \eqref{ellip-int-Frac}. We also introduce an optimal growth condition on the nonlinearity $f$ which leads to the Lipschitz continuity of the solution map.

\subsection{Existence of weak solutions}\label{sec:weaksol}

We shall denote by $(\cV_0)^\star=(\mathbb H^{s}(\Omega)\cap L_F(\Omega))^\star$ the dual of the reflexive Banach spaces $\cV_0$. Throughout the remainder of the paper,  given a reflexive Banach space $\mathbb X$ and its dual $\mathbb X^\star$, we shall denote by $\langle\cdot,\cdot\rangle_{\mathbb X^\star,\mathbb X}$ their duality map. Now we can introduce our notion of weak solution to \eqref{ellip-pro}.

\begin{definition}
 A  $u\in \cV_0$ is said to be a weak solution of  \eqref{ellip-pro} if the identity
\begin{align}\label{form-ws}
\mathcal F_D(u,v):=\int_{\Omega}(L_D)^{\frac s2}u(L_D)^{\frac s2}v\;dx+\int_{\Omega}f(x,u)v\;dx=\langle z,v\rangle_{(\cV_0)^\star,\mathcal V_0},
\end{align}
holds for every $v\in \cV_0$ and the right hand side of \eqref{form-ws} makes sense.
\end{definition}

We have the following result of existence and uniqueness of weak solution.

\begin{proposition}[\bf Existence of weak solution]\label{prop-exis}
Let Assumption \ref{assum3} hold.  Then for every $z\in (\cV_0)^\star$, \eqref{ellip-pro} has a unique weak solution $u$. In addition, if $z\in \mathbb H^{-s}(\Omega)\hookrightarrow (\mathcal V_0)^\star$, then there is a constant $C>0$ such that 
\begin{equation}\label{nor-est}
\|u\|_{\mathbb H^s(\Omega)}\le C\|z\|_{\mathbb H^{-s}(\Omega)}.
\end{equation}
\end{proposition}

\begin{proof}
The proof follows the lines of the case $L=-\Delta$ contained in \cite[Proposition 2.8]{HAntil_JPfefferer_MWarma_2016a}. We omit it for brevity.
\end{proof}

The following result gives further estimates for the difference of two solutions.

\begin{proposition} \label{pro-34}
Let Assumption \ref{assum3} hold. Let $z_1,z_2\in \mathbb H^{-s}(\Omega)\hookrightarrow (\mathcal V_0)^\star$ and $u_1,u_2\in\mathcal V_0$ be the corresponding weak solutions of \eqref{ellip-pro}. Then there is a constant $C=C(N,s,\Omega)>0$ such that
\begin{align}\label{EST1}
C\|u_1-u_2\|_{L^2(\Omega)}\le \|u_1-u_2\|_{\mathbb H^s(\Omega)}\le \|z_1-z_2\|_{\mathbb H^{-s}(\Omega)}.
\end{align} 
\end{proposition}

\begin{proof} 
Taking $v=u_1-u_2$ as a test function in \eqref{form-ws}, we get that
\begin{align*}
\int_{\Omega}|(L_D)^{\frac s2}(u_1-u_2)|^2\;dx+\int_{\Omega}\left[f(x,u_1)-f(x,u_2)\right](u_1-u_2)\;dx 
=\langle z_1-z_2,u_1-u_2\rangle_{\mathbb H^{-s}(\Omega),\mathbb H^s(\Omega)}.
\end{align*}
Since $f(x,\cdot)$ is monotone for a.e. $x\in\Omega$, we have that $\left[f(x,u_1)-f(x,u_2)\right](u_1-u_2)\ge 0$. Thus,  from the preceding identity  we can deduce that

\begin{align*}
\|u_1-u_2\|_{\mathbb H^s(\Omega)}^2=\int_{\Omega}|(L_D)^{\frac s2}(u_1-u_2)|^2\;dx\le\langle z_1-z_2,u_1-u_2\rangle_{\mathbb H^{-s}(\Omega),\mathbb H^s(\Omega)}
\le \|z_1-z_2\|_{\mathbb H^{-s}(\Omega)}\|u_1-u_2\|_{\mathbb H^{s}(\Omega)}.
\end{align*}
The above estimate together with the embedding $\mathbb H^s(\Omega) \hookrightarrow L^2(\Omega)$ imply \eqref{EST1}.
\end{proof}

\begin{remark}\label{op-rem}
{\em  From Proposition \ref{prop-exis} (its proof), we have  that for every $u\in \cV_0$ there exists a unique $A_F(u)\in (\cV_0)^\star$ such that $\mathcal F_D(u,v)=\langle A_F(u),v\rangle_{(\cV_0)^\star,\mathcal V_0}$ for every $v\in \cV_0$. This defines an operator $A_F:\cV_0\to (\cV_0)^\star$ which is hemi-continuous, strictly monotone, continuous, surjective and bounded. }
\end{remark}

Next we give further qualitative properties of the above mentioned operator.

\begin{proposition}
\label{prof:growth}
Let $A_F:\mathcal V_0\to (\cV_0)^\star$ be the surjective, continuous and bounded operator mentioned in Remark \ref{op-rem}. Then $A_F$ is also injective, hence invertible and its inverse $A_F^{-1}$ is bounded from $ (\cV_0)^\star$ into $\mathcal V_0$. In addition, if $f$ satisfies the following growth condition: there exists a constant $c\in (0,1]$ such that
\begin{align}\label{growth}
c|f(x,\xi-\eta)|\le |f(x,\xi)-f(x,\eta)|
\end{align}
for a.e. $x\in\Omega$ and for all $\xi,\eta\in\RR$, then $A_F^{-1}$ is also continuous from $ (\cV_0)^\star$ into $\mathcal V_0$. Furthermore, if $r>(2^\star)'=\frac{2N}{N+2s}$, then $A_F^{-1}: L^{r}(\Omega)\to \mathcal V_0$ and $A_F^{-1}: L^r(\Omega)\to L^p(\Omega)$ are compact for every $p\in (1,2^{\star})$.
\end{proposition}

\begin{proof}
Recall that by Remark \ref{op-rem}, the operator $A_F$ is strictly monotone. More precisely, we have that
\begin{align*}
\langle A_F(u)-A_F(v),u-v\rangle_{(\cV_0)^\star,\cV_0}=\mathcal F_D(u,u-v)-\mathcal F_D(v,u-v)>0,
\end{align*}
for all $u,v\in\mathcal V_0$ with $u\ne v$. This shows that $A_F$ is injective and hence, $A_F^{-1}$ exists. The estimate
\begin{align*}
\mathcal F_D(u,u)=\langle A_F(u),u\rangle_{(\cV_0)^\star,\cV_0}\le \|A_F(u)\|_{(\mathcal V_0)^\star}\|u\|_{\mathcal V_0},
\end{align*}
together with the coercivity of $\mathcal F_D$, that is, $\lim_{\|u\|_{\mathcal V_0}\to\infty}\frac{\mathcal F_D(u,u)}{\|u\|_{\mathcal V_0}}=\infty$ (see \cite[Proposition 2.8]{HAntil_JPfefferer_MWarma_2016a}),  imply that
\begin{align*}
\lim_{\|u\|_{\mathcal V_0}\to\infty}\|A_F(u)\|_{(\mathcal V_0)^\star}=\infty.
\end{align*}
Thus $A_F^{-1}:(\mathcal V_0)^\star\to\mathcal V_0$ is bounded.

Next, assume that the nonlinearity $f$ satisfies \eqref{growth}. Notice that it follows from \eqref{growth} that
\begin{align}\label{F1}
(f(x,\xi)-f(x,\eta))(\xi-\eta)\ge cf(x,\xi-\eta)(\xi-\eta),
\end{align}
for a.e. $x\in\Omega$ and for all $\xi,\eta\in\RR$.
The estimate \eqref{F1} together with the $(\Delta_2)$-condition \eqref{delta-2} imply that for every $u,v\in\mathcal V_0$,
\begin{align}\label{F2}
\int_{\Omega}(f(x,u)-f(x,v))(u-v)\;dx\ge \int_{\Omega}cf(x,u-v)(u-v)\;dx\ge \int_{\Omega}F(x,u-v)\;dx.
\end{align}
We show that $A_F^{-1}:(\mathcal V_0)^\star\to\mathcal V_0$ is continuous. Assume that $A_F^{-1}$ is not continuous. Then there exist a sequence $\{z_n\}_{n\in\NN}\subset (\mathcal V_0)^\star$ with $z_n\to z$ in $(\mathcal V_0)^\star$ as $n\to\infty$, and a constant $K>0$ such that
\begin{align}\label{no-cont}
\|A_F^{-1}(z_n)-A_F^{-1}(z)\|_{\mathcal V_0}\ge K\;\mbox{ for all }\;n\in\NN.
\end{align}
Let $u_n:=A_F^{-1}(z_n)$ and $u:=A_F^{-1}(z)$. Since $\{z_n\}_{n\in\NN}$ is a bounded sequence and $A_F^{-1}$ is bounded, we have that $\{u_n\}_{n\in\NN}$ is a bounded sequence in $\mathcal V_0$. Since $\mathcal V_0$ is a reflexive Banach space, by possibly passing to a subsequence if necessary, we may assume that $u_n$ converges weakly to some $v\in\mathcal V_0$ as $n\to\infty$. Since $A_F(u_n)-A_F(v)\to z-A_F(v)$ in $(\mathcal V_0)^\star$ as $n\to\infty$, and $(v-u_n)\rightharpoonup 0$ in $\mathcal V_0$ as $n\to\infty$, it follows that
\begin{align}\label{ee}
\lim_{n\to\infty}\langle A_F(u_n)-A_F(v),u_n-v\rangle_{(\cV_0)^\star,\cV_0}=0.
\end{align}
Using \eqref{F2} we get that for every $n\in\NN$,
\begin{align*}
\int_{\Omega}|(L_D)^{\frac s2}(u_n-v)|^2\;dx+\int_{\Omega}F(x,u_n-v)\;dx\le c\langle A_F(u_n)-A_F(v),u_n-v\rangle_{(\cV_0)^\star,\cV_0}.
\end{align*}
This estimate together with \eqref{ee} imply that
\begin{align*}
\lim_{n\to\infty}\int_{\Omega}|(L_D)^{\frac s2}(u_n-v)|^2\;dx=0\;\mbox{ and } \lim_{n\to\infty}\int_{\Omega}F(x,u_n-v)\;dx=0.
\end{align*}
Thus, using the $(\Delta_2)$-condition \eqref{delta-2}, we get that $u_n\to v$ in $\mathcal V_0$ as $n\to\infty$. Since $A_F$ is demi-continuous (this follows from the fact that $A_F$ is hemi-continuous, monotone and bounded by Remark \ref{op-rem}), it follows that
\begin{align*}
z_n=A_F(u_n)\rightharpoonup A_F(v)\;\mbox{ in }\; (\mathcal V_0)^\star\;\mbox{ and }\; z_n\to z=A_F(u)\;\mbox{ in }\; (\mathcal V_0)^\star\;\mbox{ as } n\to\infty.
\end{align*}
The uniqueness of the weak limit implies that $A_F(u)=z=A_F(v)$ and hence, by the injectivity of $A_F$ we get that $u=v$. We have shown that
\begin{align*}
\lim_{n\to\infty}|\|A_F^{-1}(z_n)-A_F^{-1}(z)\|_{\mathcal V_0}=\lim_{n\to\infty}\|u_n-u\|_{\mathcal V_0}=0,
\end{align*}
and this contradicts \eqref{no-cont}. 
Thus, $A_F^{-1}:(\mathcal V_0)^\star\to\mathcal V_0$ is continuous. 

Next let $1<q<2^\star$. Since the embedding $\mathcal V_0\hookrightarrow L^q(\Omega)$ is compact, then by duality, the embedding $L^{r}(\Omega)\hookrightarrow (\mathcal V_0)^\star$ is compact for every $r>(2^\star)'=\frac{2N}{N+2s}$. This, together with the fact that $A_F^{-1}:(\mathcal V_0)^\star\to\mathcal V_0$  is continuous and bounded, imply that $A_F^{-1}:L^{r}(\Omega)\to\mathcal V_0$ is compact for every $r>(2^\star)'=\frac{2N}{N+2s}$.
It remains to show that $A_F^{-1}$ is also compact as a map into $L^p(\Omega)$ for every $p\in (1,2^{\star})$. Since $A_F^{-1}$ is bounded, we have to show that the image of every bounded set $\mathcal B\subset L^r(\Omega)$ is relatively compact in $L^p(\Omega)$ for every $1<p<2^\star$. Let $\{u_n\}_{n\in\NN}$ be a sequence in $A_F^{-1}(\mathcal B)$ and $z_n:=A_F(u_n)\in\mathcal B$. Since $\mathcal B$ is bounded, it follows that $\{z_n\}_{n\in\NN}$ is bounded. Since $A_F^{-1}$ is compact as a map into $\mathcal V_0$, we have that there is a subsequence denoted again $\{z_n\}_{n\in\NN}$ such that $A_F^{-1}(z_n)\to u$ in $\mathcal V_0$ as $n\to\infty$, and hence also in $L^2(\Omega)$. We have to show that $u_n\to u$ in $L^p(\Omega)$ as $n\to\infty$. Let $p\in [2,2^{\star})$. Since $\{u_n\}_{n\in\NN}$ is bounded in $L^{2^\star}(\Omega)$, a standard interpolation inequality shows that there is $\tau\in (0,1)$ such that
\begin{align}\label{in-interpo}
\|u_n-u_m\|_{L^p(\Omega)}\le \|u_n-u_m\|_{L^2(\Omega)}^\tau\|u_n-u_m\|_{L^{2^\star}(\Omega)}^{1-\tau}\le C\|u_n-u_m\|_{L^2(\Omega)}^\tau,
\end{align}
for some constant $C>0$ independent of $n$. More precisely $\frac 1p=\frac{\tau}{2}+\frac{1-\tau}{2^\star}$.
Since $2\le p<2^\star$, a simple calculation gives that $0<\tau=\frac{2(2^\star-p)}{p(2^\star-2)}<1$.
Now, as $u_n$ converges in $L^2(\Omega)$, it follows from \eqref{in-interpo} that $\{u_n\}_{n\in\NN}$ is a Cauchy sequence in $L^p(\Omega)$ and therefore converges in $L^p(\Omega)$. Hence, $A_F^{-1}:L^r(\Omega)\to L^p(\Omega)$ is compact for every $p\in [2,2^\star)$. The case $p\in (1,2)$ follows from the embedding $L^2(\Omega)\hookrightarrow L^p(\Omega)$.  The proof is finished.
\end{proof}

We conclude this subsection with the following comment.
\begin{remark}
{\em We mention that even if the results in Proposition \ref{prof:growth} are not explicitly used further in the present article, they are important in their own and can be used to obtain some qualitative properties of solutions to associated semilinear parabolic problems. In addition, using these results, one can define the functions $J_1$ and $J_2$ (see Section \ref{s:red_prob}) on $L^p(\Omega)$ for $p\in (1,2^\star)$ and obtain some of the results. Further properties of $A_F$ will be given in Section \ref{s:red_prob} since the operator $S$ defined in \eqref{eq:SLinf} and \eqref{eq:SL2} is nothing else than the restriction of $A_{F}^{-1}$ to $L^\infty(\Omega)$ and $L^2(\Omega)$, respectively.}
\end{remark}

\subsection{Regularity of weak solutions}\label{sec-weak-sol}
The following theorem is the first main result of this section.

\begin{theorem}\label{theo-bound}
Let Assumption \ref{assum3} hold and assume that $z\in L^p(\Omega)$ with 
\begin{equation}\label{cond-p}
\begin{cases}
p>\frac{N}{2s}\;\;&\mbox{ if }\; N>2s,\\
p>1 \;\;&\mbox{ if }\; N=2s,\\
p=1\;\;&\mbox{ if }\; N<2s.
\end{cases}
\end{equation}
Then every weak solution $u$ of  \eqref{ellip-pro} belongs to $L^\infty(\Omega)$ and there is a constant $C=C(N,s,p,\Omega)>0$ such that
\begin{align}\label{inf-norm}
\|u\|_{L^\infty(\Omega)}\le C\|z\|_{L^p(\Omega)}.
\end{align}
\end{theorem}

\begin{remark}\rm
We mention that if $N=1$ and $\frac 12<s<1$, then it follows from \eqref{sobo2} that the weak solution of \eqref{ellip-pro} is globally H\"older continuous on $\bOm$ and in this case there is nothing to prove. Thus we need to prove the theorem only in the cases $N\ge 2$, or $N=1$ and $0<s\le \frac 12$.
\end{remark}

\begin{proof}[\bf Proof of Theorem \ref{theo-bound}]
The proof of this theorem also follows along the lines of the proof of the corresponding result for the case $L=-\Delta$ contained in \cite[Theorem 2.9]{HAntil_JPfefferer_MWarma_2016a}. It can also be obtained by taking $u_2=0$ and $z_2=0$ in the proof of Proposition \ref{pro-dif-sol} below and in that case the growth condition \eqref{growth} on $f$ is not needed. 
\end{proof}

We notice that the proof of \cite[Theorem 2.9]{HAntil_JPfefferer_MWarma_2016a} (hence, the proof of Theorem \ref{theo-bound}) uses the following result \cite[Lemma B.1]{kinderlehrer1980} which will be also used in the proof of Proposition \ref{pro-dif-sol} below.
\begin{lemma}\label{lem-01}
Let $\Xi = \Xi(t)$ be a nonnegative, non-increasing function on a half line $t\ge k_0\ge 0$ such that there are positive constants $c, \alpha$ and $\delta$ ($\delta >1$) with
\begin{equation*}
\Xi(h) \le c(h-k)^{-\alpha}\Xi(k)^{\delta}\mbox{ for }  h>k\ge k_0.
\end{equation*}
Then
\begin{equation*}
\Xi(k_0+d) = 0\quad \mbox{ with }\quad d^{\alpha}= c \Xi(k_0)^{\delta -1}2^{\alpha\delta/(\delta -1)}.
\end{equation*}
\end{lemma}

Next we give an $L^\infty$-estimate for the difference of two solutions which is the second main result of this section.

\begin{proposition}\label{pro-dif-sol}
Assume that Assumption \ref{assum3} holds and that $f$ satisfies the growth condition \eqref{growth}. Let $z_1,z_2\in L^p(\Omega)$ with $p$ as in \eqref{cond-p} and let $u_1,u_2\in\mathcal V_0\cap L^\infty(\Omega)$ be the corresponding weak solutions of \eqref{ellip-pro}. Then there is a constant $C=C(N,p,s,\Omega)>0$ such that
\begin{align}\label{EST2}
\|u_1-u_2\|_{L^\infty(\Omega)}\le C\|z_1-z_2\|_{L^p(\Omega)}.
\end{align}
\end{proposition}

\begin{proof}
 We prove the proposition in two steps.
 
{\bf Step 1}.
Let $k\ge 0$. Set $v:=u_1-u_2$ and $v_k:=(|v|-k)^+\sgn(v)$. We claim that $v_k\in\mathcal V_0$ and 
\begin{align}\label{CLAIM}
c\mathcal F_D(v_k,v_k)\le \mathcal F_D(u_1,v_k)-\mathcal F_D(u_2,v_k),
\end{align}
for every $k\ge 0$, where $c\in (0,1]$ is the constant appearing in \eqref{growth}. Indeed, using \cite[Lemma 2.7]{War} we get that $v_k\in \cV_0$.
Let $A_k:=\{x\in\Om:\;|v(x)|\ge k\}$, $A_k^+:=\{x\in\Om:\;v(x)\ge k\}$ and $A_k^-:=\{x\in\Om:\;v(x)\le -k\}$ so that $A_k=A_k^+\cup A_k^-$. It follows from the representation \eqref{int-rep}  that
\begin{align}\label{Aa1}
 &\mathcal F_D(u_1,v_k)-\mathcal F_D(u_2,v_k)\notag\\
 =&\int_{\Omega}(L_D)^{\frac s2}(u_1-u_2)(L_D)^{\frac s2}(v_k)\;dx+\int_{\Omega}\Big(f(x,u_1)-f(x,u_2)\Big)v_k(x)\;dx\notag\\
 =&\frac 12\int_{\Omega}\int_{\Omega}\Big((u_1-u_2)(x)-(u_1-u_2)(y)\Big)\Big(v_k(x)-v_k(y)\Big)K_s^L(x,y)\;dxdy\notag\\
 &+\int_{\Omega}\kappa_s(x)(u_1-u_2)(x)v_k(x)\;dx+\int_{\Omega}\Big(f(x,u_1)-f(x,u_2)\Big)v_k(x)\;dx.
\end{align}
Proceeding exactly as the proof of  \cite[Theorem 2.9]{HAntil_JPfefferer_MWarma_2016a} we get that
\begin{align}\label{AA1}
&\frac 12\int_{\Omega}\int_{\Omega}\Big((u_1-u_2)(x)-(u_1-u_2)(y)\Big)\Big(v_k(x)-v_k(y)\Big)K_s^L(x,y)\;dxdy\notag\\
& +\int_{\Omega}\kappa_s(x)(u_1-u_2)(x)v_k(x)\;dx\notag\\
 \ge &\frac 12\int_{\Omega}\int_{\Omega}|v_k(x)-v_k(y)|^2K_s^L(x,y)\;dxdy +\int_{\Omega}\kappa_s(x)|v_k(x)|^2\;dx
 =\int_{\Omega}|(L_D)^{\frac s2}(v_k)|^2\;dx.
\end{align}
We notice that
\begin{align}\label{AA2}
\int_{\Omega}\Big(f(x,u_1)-f(x,u_2)\Big)v_k\;dx
=\int_{\Omega}cf(x,v_k)v_k\;dx +\int_{\Omega}\Big[f(x,u_1)-f(x,u_2)-cf(x,v_k)\Big]v_k\;dx,
\end{align}
where $c\in (0,1]$ is the constant appearing in \eqref{growth}. For a.e. $x\in A_k^+$, we have that
\begin{align*}
cf(x,v_k(x))
 =cf(x,u_1(x)-u_2(x)-k)\le cf(x,u_1(x)-u_2(x)) 
 \le f(x,u_1(x))-f(x,u_2(x)).
\end{align*}
Multiplying this inequality with $v_k(x)\ge 0$ gives for a.e. $x\in A_k^+$,
\begin{align}\label{Bbb1}
\Big[f(x,u_1(x))-f(x,u_2(x))-cf(x,v_k(x))\Big]v_k(x)\ge 0.
\end{align}
Similarly, for a.e. $x\in A_k^-$, we have that
\begin{align*}
cf(x,v_k(x))
=cf(x,u_1(x)-u_2(x)+k)\ge cf(x,u_1(x)-u_2(x)) 
\ge f(x,u_1(x))-f(x,u_2(x)).
\end{align*}
Hence, multiplying this inequality with $v_k(x)\le 0$, we get that for a.e. $x\in A_k^-$,
\begin{align}\label{Bbb2}
\Big[f(x,u_1(x))-f(x,u_2(x))-cf(x,u_k(x))\Big]v_k(x)\ge 0.
\end{align}
Combining \eqref{AA2}, \eqref{Bbb1} and \eqref{Bbb2} we get that for every $k\ge 0$,
\begin{align}\label{B3}
\int_{\Omega}\Big(f(x,u_1)-f(x,u_2)\Big)v_k\;dx\ge c\int_{\Omega}f(x,v_k)v_k\;dx.
\end{align}
Now it follows from \eqref{Aa1}, \eqref{AA1} and \eqref{B3} that
\begin{align*}
\mathcal F_D(u_1,v_k)-\mathcal F_D(u_2,v_k)\ge& \int_{\Omega}|(L_D)^{\frac s2}(v_k)|^2\;dx +c\int_{\Omega}f(x,v_k)v_k\;dx\\
\ge& c\left(\int_{\Omega}|(L_D)^{\frac s2}(v_k)|^2\;dx +\int_{\Omega}f(x,v_k)v_k\;dx\right)
=c\mathcal F_D(v_k,v_k),
\end{align*}
and we have shown the claim \eqref{CLAIM}.

{\bf Step 2.} It follows from \eqref{CLAIM} that there is a constant $C>0$ such that for every $k\ge 0$, 
\begin{align}\label{claim1}
C\|v_k\|_{L^{2^\star}(\Omega)}^2
\le c\|v_k\|_{\mathbb H^s(\Omega)}^2\le c\mathcal F_D(v_k,v_k)\le \mathcal F_D(u_1,v_k)-\mathcal F_D(u_2,v_k)
=\int_{\Omega}(z_1-z_2)v_k\;dx.
\end{align}
Let $p_1\in [1,\infty]$ be such that $\frac{1}{p}+\frac{1}{2^\star}+\frac{1}{p_1}=1$ where  $2^\star=\frac{2N}{N-2s}>2$. Since $p>\frac{N}{2s}=\frac{2^\star}{2^\star-2}$, we have that
\begin{align}\label{eq-B}
\frac{1}{p_1}=1-\frac{1}{2^\star}-\frac{1}{p}>\frac{2^\star}{2^\star}-\frac{1}{2^\star}-\frac{2^\star-2}{2^\star}=\frac{1}{2^\star}\Longrightarrow p_1<2^\star.
\end{align}
Using the classical H\"older inequality and noticing that $v_k=0$ on $\Omega\setminus A_k$, we get from \eqref{claim1} that there is a constant $C=C(N,s,p)>0$ such that for every $k\ge 0$,
\begin{align}\label{est}
c\mathcal F_D(v_k,v_k)\le \int_{\Omega}(z_1-z_2)v_k\;dx=\int_{A_k}(z_1-z_2)v_k\;dx\le& \|z_1-z_2\|_{L^p(\Omega)}\|v_k\|_{L^{2^\star}(\Omega)}\|\chi_{A_k}\|_{L^{p_1}(\Omega)},
\end{align}
where $\chi_{A_k}$ denotes the characteristic function of the set $A_k$.
Using \eqref{claim1}, \eqref{est}, \eqref{sobo1} and the fact that $\int_{\Om}f(x,v_k)v_k\;dx\ge 0$, we get that there are two constants $C,C_1>0$ such that for every $k\ge 0$,
\begin{align*}
C\|v_k\|_{L^{2^\star}(\Omega)}^2 \le c\|v_k\|_{\mathbb H^{s}(\Omega)}^2\le c\mathcal F_D(v_k,v_k)
\le C_1 \|z_1-z_2\|_{L^p(\Omega)}\|v_k\|_{L^{2^\star}(\Omega)}\|\chi_{A_k}\|_{L^{p_1}(\Omega)}.
\end{align*}
This estimate implies that there is a constant $C>0$ such that for every $k\ge 0$,
\begin{align}\label{est2}
\|v_k\|_{L^{2^\star}(\Omega)}\le C \|z_1-z_2\|_{L^p(\Omega)}\|\chi_{A_k}\|_{L^{p_1}(\Omega)}.
\end{align}
Let $h>k$. Then $A_h\subset A_k$ and on $A_h$ we have $|v_k|\ge h-k$. It follows from \eqref{est2} that for every $h>k\ge 0$,
\begin{align}\label{B1}
\|\chi_{A_h}\|_{L^{2^\star}(\Omega)}\le C (h-k)^{-1}\|z_1-z_2\|_{L^p(\Omega)}\|\chi_{A_k}\|_{L^{p_1}(\Omega)}.
\end{align}
Let $\delta:=\frac{2^\star}{p_1}>1$ by \eqref{eq-B}. Then using the H\"older inequality again we get that there is a constant $C>0$ such that for every $k\ge 0$, 
\begin{align}\label{B2}
\|\chi_{A_k}\|_{L^{p_1}(\Omega)}\le C \|\chi_{A_k}\|_{L^{2^\star}(\Omega)}^\delta.
\end{align}
It follows from \eqref{B1} and \eqref{B2} that there is a constant $C>0$ such that for every $h>k\ge 0$,
\begin{align*}
\|\chi_{A_h}\|_{L^{2^\star}(\Omega)}\le C (h-k)^{-1}\|z_1-z_2\|_{L^p(\Omega)}\|\chi_{A_k}\|_{L^{2^\star}(\Omega)}^\delta.
\end{align*}
Using Lemma \ref{lem-01} with $\Xi(k)=\|\chi_{A_k}\|_{L^{2^\star}(\Omega)}$, we get that there is a constant $C_2>0$ such that
\begin{align*}
\|\chi_{A_K}\|_{L^{2^\star}(\Omega)}=0\;\mbox{ with }\;K=CC_2\|z_1-z_2\|_{L^p(\Omega)}.
\end{align*}
We have shown \eqref{EST2} and the proof is finished.
\end{proof}

\begin{remark}
{\rm We mention that all the results given in Proposition \ref{pro-dif-sol} remain true if one replaces the growth condition \eqref{growth} with the following local Lipschitz continuity condition: for all $M>0$ there exists a constant $L_{M}>0$ such that $f$ satisfies 
	\begin{align}\label{Lip-cond}
	|f(x,\xi)-f(x,\eta)|\leq L_{M}|\xi-\eta|,
	\end{align}
	for a.e. $x\in\Omega$ and $\xi,\eta\in\mathbb{R}$ with $|\eta|, |\xi|\leq M$. Of course in that case the constant $C$ that appears in \eqref{EST2} also depends on the Lipschitz constant $L_M$.
     A condition such as \eqref{Lip-cond} is needed to prove the $\mathbb{H}^{2s+\beta}(\Omega)$
     regularity for $u$ provided that $z \in \mathbb{H}^\beta(\Omega)$ where $0\le\beta < 1$
     (see \cite[Corollary~2.15]{HAntil_JPfefferer_MWarma_2016a}). This higher regularity is 
     important for finite element error estimates as shown in \cite[Section~4]{HAntil_JPfefferer_MWarma_2016a}.
}
\end{remark}

\subsection{The case of the fractional Laplace operator}\label{sec-int-frac}
In this section we consider the semilinear elliptic problem \eqref{ellip-int-Frac}. Firstly, 
the fractional Laplace operator $(-\Delta)^s$ is given formally for $x\in\RR^N$ by
\begin{align}\label{IFL}
(-\Delta)^su(x)=C_{N,s}\mbox{P.V.}\int_{\RR^N}\frac{u(x)-u(y)}{|x-y|^{N+2s}}\;dy
=C_{N,s}\lim_{\varepsilon\downarrow 0}\int_{\{y\in\RR^N:|x-y|\ge \varepsilon\}}\frac{u(x)-u(y)}{|x-y|^{N+2s}}\;dy,
\end{align}
whenever the limit exists, and $C_{N,s}$ is a normalization constant depending only on $N$ and $s$. We refer to \cite{Caf3,Caf1,NPV,W-Nodea} for the class of functions for which the limit exists and for further properties and applications of this operator.

Secondly, in order to give our notion of solution to \eqref{ellip-int-Frac} we need  the fractional order Sobolev space
\begin{align*}
H_0^{s}(\bOm):=\{u\in H^s(\RR^N):\; u=0\;\mbox{ in }\;\RR^N\setminus\Omega\}.
\end{align*}
For every $0<s<1$, we have that  $H_0^{s}(\bOm)$ endowed with the norm
\begin{align*}
\left(\int_{\RR^N}\int_{\RR^N}\frac{|u(x)-u(y)|^2}{|x-y|^{N+2s}}\;dxdy\right)^{\frac 12},
\end{align*}
is a Hilbert space, and we shall denote by $H^{-s}(\bOm)$ its dual. It is well known (see e.g. \cite{NPV}) that the embedding \eqref{inj1} holds with $H_0^s(\Omega)$ replaced by $H_0^s(\bOm)$. We notice that there is a priori no obvious inclusion between $H_0^s(\Omega)$ and $H_0^s(\bOm)$. In fact for an arbitrary bounded open set, the two spaces are different since by \cite{Val}, $\mathcal D(\Omega)$ is not always dense in $H_0^{s}(\bOm)$. But if $\Omega$ has a Lipschitz continuous boundary, then by  \cite{Val}, $\mathcal D(\Omega)$ is dense in $H_0^{s}(\bOm)$. In that case we have that $H_0^s(\Omega)=H_0^s(\bOm)$ if $\frac 12<s<1$. For $0<s\le\frac 12$, even if $\Omega$ is smooth, the two spaces do not coincide. In fact if $\Omega$ has a Lipschitz continuous boundary and $0<s\le \frac 12$, then by Remark \ref{rem-sob}, $H_0^s(\Omega)=H^s(\Omega)$. This shows that the constant function $1$ belongs to $H^s(\Omega)=H_0^s(\Omega)$, but clearly $1\not\in H_0^s(\bOm)$. 

Here is our notion of solution to the system \eqref{ellip-int-Frac}.

\begin{definition}
A $u\in H_0^{s}(\bOm)$ is said to be a weak solution of \eqref{ellip-int-Frac} if the identity 
\begin{align*}
\frac{C_{N,s}}{2}\int_{\RR^N}\int_{\RR^N}\frac{(u(x)-u(y))(v(x)-v(y))}{|x-y|^{N+2s}}\;dxdy+\int_{\Omega}f(x,u)v\;dx=\langle z,v\rangle_{H^{-s}(\bOm),H_0^{s}(\bOm)},
\end{align*}
holds for every $v\in  H_0^{s}(\bOm)$ and the right hand side makes sense.
\end{definition}

Thirdly, notice that the operator $(-\Delta)^s$ defined in \eqref{IFL} does not incorporate a boundary condition.  We let $(-\Delta)_D^s$ be the selfadjoint operator on $L^2(\Omega)$ associated with the bilinear form 
\begin{align*}
\mathbb E_D(u,v):=\frac{C_{N,s}}{2}\int_{\RR^N}\int_{\RR^N}\frac{(u(x)-u(y))(v(x)-v(y))}{|x-y|^{N+2s}}\;dxdy,\;u,v\in H_0^s(\bOm),
\end{align*}
in the sense that 
\begin{align*}
\begin{cases}
D((-\Delta)_D^s)=\left\{u\in H_0^s(\bOm),\;\exists\;w\in L^2(\Omega),\;\mathbb E_D(u,v)=(w,v)_{L^2(\Omega)}\;\forall v\in H_0^s(\bOm)\right\}\\
(-\Delta)_D^su=w.
\end{cases}
\end{align*}
Then $(-\Delta)_D^s$ is the realization in $L^2(\Omega)$ of $(-\Delta)^s$ with the Dirichlet exterior condition $u=0$ in $\RR^N\setminus\Omega$. More precisely, we have that
\begin{align*}
D((-\Delta)_D^s)=\{u\in H_0^s(\bOm):\; (-\Delta)_D^su\in L^2(\Omega)\},\;\; (-\Delta)_D^su=(-\Delta)^su.
\end{align*}
With this setting, the system \eqref{ellip-int-Frac} can be rewritten as
\begin{align*}
(-\Delta)_D^su+f(x,u)=z\quad\mbox{ in }\;\Omega.
\end{align*}
We also mention that even if we take $a_{ij}=\delta_{ij}$, that is, $L=-\Delta$, the operators $(L_D)^s$ and $(-\Delta)_D^s$ are different. More precisely their eigenvalues and eigenfunctions are different. We refer to \cite{BWZ,SeVa} for more details on this topic.

Finally, we notice the following.
\begin{remark}
{\rm 
Letting $\widetilde \cV_0:=H_0^s(\bOm)\cap L_F(\Omega)$, then all the results in Sections \ref{sec:weaksol} and \ref{sec-weak-sol} hold for the system \eqref{ellip-int-Frac} if one replaces in the statements and proofs $\cV_0$, $\mathbb H^s(\Omega)$ and $\mathbb H^{-s}(\Omega)$ by $\widetilde \cV_0$,  $H_0^s(\bOm)$ and $H^{-s}(\bOm)$, respectively and the form $\mathcal F_D$ by
\begin{align}\label{Form-IFL}
\mathbb F_D(u,v)=\frac{C_{N,s}}{2}\int_{\RR^N}\int_{\RR^N}\frac{(u(x)-u(y))(v(x)-v(y))}{|x-y|^{N+2s}}\;dxdy+\int_{\Omega}f(x,u)v\;dx,
\end{align}
for $u,v\in H_0^s(\bOm)$.
All the proofs follow similarly with very minor modifications.  In fact in the proofs of all the results in Sections \ref{sec:weaksol} and \ref{sec-weak-sol}, the main tool is the representation \eqref{int-rep} of the operator $(L_D)^s$. Here also, using the definition of $(-\Delta)_D^s$ and \eqref{Form-IFL} we get that for $u,v\in H_0^s(\bOm)$ (using that $u=v=0$ in $\RR^N\setminus\Omega$),
\begin{align*}
\langle (-\Delta)_D^su,v\rangle_{H^{-s}(\bOm),H_0^s(\bOm)}=&\frac{C_{N,s}}{2}\int_{\RR^N}\int_{\RR^N}\frac{(u(x)-u(y))(v(x)-v(y))}{|x-y|^{N+2s}}\;dxdy\\
=&\frac{C_{N,s}}{2}\int_{\Omega}\int_{\Omega}\frac{(u(x)-u(y))(v(x)-v(y))}{|x-y|^{N+2s}}\;dxdy\\
&+\frac{C_{N,s}}{2}\int_{\Omega}\int_{\RR^N\setminus\Omega}\frac{u(y)v(y)}{|x-y|^{N+2s}}\;dxdy
+\frac{C_{N,s}}{2}\int_{\RR^N\setminus\Omega}\int_{\Omega}\frac{u(x)v(x)}{|x-y|^{N+2s}}\;dxdy\\
=&\frac{C_{N,s}}{2}\int_{\Omega}\int_{\Omega}\frac{(u(x)-u(y))(v(x)-v(y))}{|x-y|^{N+2s}}\;dxdy
+\int_{\Omega}\tilde\kappa_s(x)u(x)v(x)\;dx,
\end{align*} 
where we have set
\begin{align*}
\tilde\kappa_s(x)=C_{N,s}\int_{\RR^N\setminus\Omega}\frac{dy}{|x-y|^{N+2s}}.
\end{align*}
This shows that we have a similar representation as in \eqref{int-rep} where in the present situation the kernel $K_s^L(x,y)$ and $\kappa_s(x)$ are replaced by $C_{N,s}|x-y|^{-N-2s}$ and $\tilde\kappa_s(x)$, respectively. 
}
\end{remark}

\subsection{An Example}\label{s:ex}
We conclude this section with the following example.

\begin{example}\label{exam} \rm 
Let $q\in [1,\infty)$ and let $b: \Omega\to (0,\infty)$ be a  function in $L^\infty(\Omega)$ such that $b(x)> 0$ for a.e. $x\in\Omega$. Define $f:\Omega\times\mathbb R\to\mathbb R$ by $f(x,t)=b(x)|t|^{q-1}t$. It is clear that $f$ satisfies Assumption \ref{assum2} and the associated function $F:\Omega\times\mathbb R\to [0,\infty)$ is given by $F(x,t)=\frac{1}{q+1}b(x)|t|^{q+1}$. For a.e. $x\in\Omega$, the inverse $\widetilde f(x,\cdot)$ of $f(x,\cdot)$ is given by $\widetilde f(x,t)=\left(b(x)\right)^{-\frac 1q}|t|^{\frac{1-q}{q}}t$. Therefore, the complementary function $\widetilde F$ of $F$ is given by $\widetilde F(x,t)=\frac{q}{q+1}\left(b(x)\right)^{-\frac 1q}|t|^{\frac{q+1}{q}}$. Hence,
\begin{align*}
tf(x,t)=(q+1) F(x,t)\;\mbox{ and }\; t\widetilde f(x,t)=\frac{q+1}{q}\widetilde F(x,t),
\end{align*}
and we have shown that Assumption \ref{assum3} is also satisfied. 

Next, let us show that $f$ satisfies the growth condition \eqref{growth}.  If
$\xi=0$ or $\eta=0$ then the assertion \eqref{growth} is obvious since
$f(x,0)=0$.
  Hence, we assume that $\xi\not=0$ and $\eta\not=0$. Moreover, since
  $|f(x,-\gamma)|=|f(x,\gamma)|$ we may assume
  without loss of generality that $|\eta|\geq|\xi|$.
  Hence, there exists $\alpha\in\RR$ with $|\alpha|\geq 1$ such that
  $\eta=\alpha\xi$. It follows that
  \begin{align*}
|f(x,\xi-\eta)| =|b(x)|\cdot|\xi-\alpha\xi|^q=|b(x)|\cdot|1-\alpha|^q|\xi|^q
\end{align*}
  and
\begin{align*}
    |f(x,\xi)-f(x,\eta)| =&
     |b(x)|\cdot||\xi|^q\sgn(\xi)-|\alpha|^q|\xi|^q\sgn(\alpha)\sgn(\xi)| \\
        =& |b(x)|\cdot|\xi|^q\cdot|1-|\alpha|^q\sgn(\alpha)|.
  \end{align*}
  The proof is done if $c|1-\alpha|^q\leq|1-|\alpha|^q\sgn(\alpha)|$
  for all $\alpha\in\RR\setminus (-1,1)$. For $\alpha=1$ this inequality is obvious.
  For $\alpha>1$ this inequality is equivalent to
  \begin{align*}
 c(\alpha-1)^q\leq\alpha^q-1 \Longleftrightarrow c\leq \frac{\alpha^{q}-1}{(\alpha-1)^q}=g(\alpha),
  \end{align*}
  where $g:[1,\infty)\to\R$ is given by $g(x):=(x^q-1)/(x-1)^q$. Differentiating shows that
  $g'(x)\leq 0$, hence,
  \[\inf_{x>1} g(x)=\lim_{x\to \infty} g(x)=1\geq c.\]
  To finish, we prove
  the case $\alpha\leq -1$. In this case, we have to show that
  \begin{align*}
 c(1+|\alpha|)^q\leq 1+|\alpha|^q \Leftrightarrow
     c\leq \frac{1+|\alpha|^q}{(1+|\alpha|)^q}=h(|\alpha|),
  \end{align*}
  where $h:[1,\infty)\to\R$ is given by $h(x)=(1+x^q)/(1+x)^q$. Differentiating shows that
  $h'(x)\geq 0$, hence,
  \begin{align*}
\inf_{x\geq 1} h(x)=h(1)=c,
  \end{align*}
  and this completes the proof of \eqref{growth}.
In particular, we have that $f$ also satisfies the condition \eqref{Lip-cond}.
\end{example}

\section{Optimal control problem: existence}\label{s:red_prob}

Now as the state equation \eqref{ellip-pro} has a unique solution, it follows that the 
{\it control-to-state map} (solution map) 

\begin{align}\label{eq:SLinf}
 S : L^\infty(\Omega) \rightarrow \mathbb{H}^s(\Om) \cap L^\infty(\Omega), \;\; z \mapsto S(z) = u, 
\end{align}
is well defined. 
We notice that $S$ is also well defined as a map from $L^p(\Omega)$ into $\mathbb{H}^s(\Om) \cap L^\infty(\Omega)$ where $p$ is as in \eqref{cond-p}. 
We further remark that we have defined $S$ on $L^\infty(\Om)$ but we shall see below 
that $J_2$ will be defined over $L^2(\Om)$. This mismatch leads to the well-established
concept of the so-called 2-norm discrepancy \cite[Section~4]{MR2583281} 
while studying the second order sufficient conditions in Section~\ref{s:oc}. Remarkably,  for enough certain $(N,s)$ pairs we can set $p=2$ and define 
 \begin{equation}\label{eq:SL2}
   S : L^2(\Om) \rightarrow \mathbb{H}^s(\Om) \cap L^\infty(\Omega),
 \end{equation}
and as a result we can avoid the 2-norm discrepancy. This is due to the following result.

\begin{lemma}\label{lem:nodisp}
Let $0<s<1$ and $N\in\NN$. If $N<4s$, then one can take $p=2$ in \eqref{cond-p} to obtain all the results in Theorem \ref{theo-bound} and Proposition \ref{pro-dif-sol}.
\end{lemma}
 \begin{proof}
  The proof is a direct consequence of the definition of $p$ in \eqref{cond-p}.
 \end{proof} 
 If $(N,s)$ fulfills the conditions of Lemma~\ref{lem:nodisp}, then we define
 $S$ as in \eqref{eq:SL2}, otherwise we let $S$ be as in \eqref{eq:SLinf}. In order 
 to accommodate both cases \eqref{eq:SLinf} and \eqref{eq:SL2}, we define
 \begin{equation}\label{eq:SL2Linf}
   S : L^{\widetilde{p}}(\Om) \rightarrow \mathbb{H}^s(\Om) \cap L^\infty(\Omega), 
   \quad \mbox{ where } \quad 
   \widetilde{p} = \left\{ \begin{array}{ll}
                             2  & \mbox{if Lemma~\ref{lem:nodisp} holds}, \\
                             +\infty & \mbox{otherwise} . 
                           \end{array}
 \right.
 \end{equation} 
We shall see that when $\widetilde{p} = 2$ we do not have the 2-norm discrepancy. 

We begin by showing that under the growth assumption \eqref{growth}, the mapping $S$ is Lipschitz continuous.

\begin{lemma}[\bf $S$ is Lipschitz continuous]
\label{lem:S_Lip}
Let Assumption~\ref{assum3} hold. Let $z_1,z_2\in \mathbb H^{-s}(\Omega)$ and $u_1,u_2\in\mathbb{H}^s(\Om)$ be the corresponding weak solutions of \eqref{ellip-pro}. Then there is a constant $C=C(N,s,\Omega)>0$ such that
\begin{align}\label{EST1-1}
\|u_1-u_2\|_{L^2(\Omega)}\le C\|u_1-u_2\|_{\mathbb H^s(\Omega)}\le C\|z_1-z_2\|_{\mathbb H^{-s}(\Omega)}.
\end{align} 
In addition, let $p$ be as in \eqref{cond-p}.
Let Assumption~\ref{assum3} hold and assume that $f$ satisfies the growth condition \eqref{growth}. Let $u_1, u_2 \in \mathbb{H}^s(\Om) \cap L^\infty(\Omega)$
be the weak solutions to \eqref{ellip-pro} with right hand sides $z_1$ and $z_2$ in $L^p(\Omega)$, 
respectively. Then  there is a constant $C=C(N,p,s,\Omega)>0$ such that
\begin{equation}\label{Lip}
\|u_1-u_2\|_{L^\infty(\Omega)} + \|u_1-u_2\|_{\mathbb H^s(\Omega)} \le C \|z_1-z_2\|_{L^p(\Omega)}.
\end{equation}
\end{lemma}

\begin{proof}
 Firstly,  \eqref{EST1-1} is due to Proposition~\ref{pro-34}. 
 In this case one does not need the growth assumption \eqref{growth} on $f$. 
 It follows from \eqref{EST1} that there is a constant $C>0$ such that
\begin{align*}
 \|u_1-u_2\|_{\mathbb H^s(\Omega)} \le C \|u_1-u_2\|_{\mathbb H^{-s}(\Omega)}\le C \|z_1-z_2\|_{L^p(\Omega)}.
\end{align*}
Secondly, since by assumption $f$ satisfies \eqref{growth}, it follows from \eqref{EST2} in Proposition \ref{pro-dif-sol} that
\begin{align*}
\|u_1-u_2\|_{L^\infty(\Omega)} \le C \|z_1-z_2\|_{L^p(\Omega)}.
\end{align*}
Now \eqref{Lip} follows from the above two estimates. The estimate \eqref{Lip} shows that $S$ is Lipschitz continuous as a map from $L^p(\Omega)$ into $\mathbb H^s(\Omega)\cap L^\infty(\Omega)$ and hence,  as a map from $L^\infty(\Omega)$ into $\mathbb H^s(\Omega)\cap L^\infty(\Omega)$. 
\end{proof}

    \begin{remark}\label{rem:pvstp}
        {\rm In view of Lemma~\ref{lem:nodisp} we can replace $p$ by $\widetilde{p}$ 
        in Lemma~\ref{lem:S_Lip}.}
    \end{remark}

Now we recall the cost functional from \eqref{cost}, i.e., $J(u,z) := J_1(u) + J_2(z)$. We let $J_1 : L^2(\Omega) \rightarrow (-\infty,\infty]$ and $J_2 : L^2(\Omega) \rightarrow (-\infty,\infty]$, and as a result we can write the reduced minimization problem
\begin{equation}\label{RPS}
 \min_{z\in Z_{ad}} \mathcal{J}(z) := J_1(S(z)) + J_2(z).
\end{equation}
Next we state the existence of solutions to \eqref{RPS} and equivalently to the system \eqref{cost}--\eqref{ctrl}.

\begin{theorem}[\bf Existence of optimal control]
\label{thm:exist}
Let the assumptions of Theorem~\ref{theo-bound} hold. Assume in addition that $f$ satisfies the growth condition \eqref{growth} and that
\begin{align}\label{f-L2}
f(\cdot,w(\cdot))\in L^2(\Omega)\;\mbox{ for every } w\in L^\infty(\Omega).
\end{align}
Then under the following assumptions on $J_1$ and $J_2$,
\begin{enumerate}[$(i)$]
\item $J_1 :L^2(\Omega) \rightarrow [0,+\infty]$ is proper and lower semicontinuous;
 \item $J_2 : L^2(\Omega) \rightarrow (-\infty,+\infty]$ is proper, convex, 
 lower-semicontinuous and bounded from below;
\end{enumerate} 
the minimization problem \eqref{RPS} admits a solution. 
\end{theorem}

\begin{proof}
We begin by noticing that $\mathcal{J}$ is bounded from below. Therefore, the infimum 
\[
 j := \inf_{z\in\Zad} \mathcal{J}(z)
\]
exists. Let $\{(u_n,z_n)\}_{n\in\NN}$ be a minimizing sequence, that is, 
$z_n\in\Zad$ and $u_n = S(z_n)$, for $n\in\NN$, are such that $\mathcal{J}(z_n)\rightarrow j$ 
as $n\rightarrow\infty$. 

Notice that $\Zad\subset L^\infty(\Omega) \subset L^p(\Omega)$ for every $1< p<\infty$. Let $p>\frac{N}{2s}$.  Since $z_n\in\Zad$ we have that $\|z_n\|_{L^p(\Omega)}\le \|z_n\|_{L^\infty(\Omega)}\le \max\{\|z_a\|_{L^\infty(\Omega)},\|z_b\|_{L^\infty(\Omega)}\}$.
Since $L^p(\Omega)$ is a reflexive Banach space,  we have that by taking a subsequence if necessary, we may assume that
$\{z_n\}_{n\in\NN}$ converges weakly in $L^p(\Omega)$ to 
some $\bar{z} \in \Zad$, i.e., 
$ z_n \rightharpoonup \bar{z} \quad \mbox{as} \quad n \rightarrow \infty $. 
This $\bar{z}$ is the candidate for our optimal control. Notice that the aforementioned argument covers both cases of $\widetilde{p}$ defined in \eqref{eq:SL2Linf}.
% where in the case $\widetilde p=\infty$, we have a weak-star ($\rightharpoonup^\star$) convergence.

Next we shall show that the state $\{u_n\}_{n\in\NN}$ converges, as $n\to\infty$,  to some $\bar{u}$ 
in a suitable sense and $(\bar{u},\bar{z})$ satisfies the state equation. Towards this 
end we introduce the sequence 
\[
 b_n(\cdot) := f(\cdot,u_n(\cdot)), \quad n \in \NN . 
\]
Since 

\begin{align}\label{u-boun}
\|u_n\|_{L^\infty(\Omega)} \le C\|z_n\|_{L^p(\Omega)}\le C \max\{\|z_a\|_{L^\infty(\Omega)},\|z_b\|_{L^\infty(\Omega)}\},
\end{align}
it follows from the fact that $f(x,\cdot)$ is increasing  and  \eqref{f-L2} that 
$\{b_n\}_{n\in\NN}$ is bounded in  $L^2(\Omega)$. As a consequence, taking a subsequence if necessary, we may assume that 
$b_n \rightharpoonup b$ in $L^2(\Omega)$ as $n\to\infty$. 
Notice that for every $n\in\NN$,  $u_n$ satisfies the identity
\begin{align}\label{eq-B-2}
 \int_{\Omega}(L_D)^{\frac s2}u_n (L_D)^{\frac s2}v\;dx 
     =\int_{\Omega} B_nv\;dx, \quad \forall \;v \in \cV_0, 
\end{align}
where $B_n:=-f(\cdot,u_n) + z_n$ converges weakly in $L^2(\Omega)$ to $-b+\bar{z}$ as $n\to\infty$. Taking $v=u_n$ as a test function in \eqref{eq-B-2}, we get that there is a constant $C>0$ (independent of $n$) such that

\begin{align*}
\|u_n\|_{\mathbb H^s(\Omega)}\le C\|B_n\|_{L^2(\Omega)}.
\end{align*}
Since $\{B_n\}_{n\in\NN}$ is a bounded sequence in $L^2(\Omega)$, it follows from the preceding estimate that $\{u_n\}_{n\in\NN}$ is bounded in $\mathbb H^s(\Omega)$, hence as before, we may assume that $u_n\rightharpoonup \bar u$ in $\mathbb H^s(\Omega)$ as $n\to\infty$ and thus, strongly in $L^2(\Omega)$ since the embedding $\mathbb H^s(\Omega)\hookrightarrow L^2(\Omega)$ is compact. 
Since by \eqref{u-boun}, $\{u_n\}_{n\in\NN}$ is uniformly bounded in $L^\infty(\Omega)$ and by the uniqueness of the limit,
we can deduce that $\bar u\in L^\infty(\Omega)$. We have already shown that $\bar u\in\mathbb H^s(\Omega)$.
Thus, using \eqref{f-L2} and  \eqref{delta-2}, we get that 
\begin{align*}
\int_{\Omega}F(x,\bar u)\;dx\le \int_{\Omega}\bar uf(x,\bar u)\;dx<\infty.
\end{align*}
This implies that $\bar u\in L_F(\Omega)$ and thus, $\bar u\in\mathcal V_0=\mathbb H^s(\Omega)\cap L_F(\Omega)$.
Next, we show that $\bar{u}$ is the weak solution associated with $\bar{z}$.  Firstly, since $u_n$ converges a.e. to $\bar u$ in $\Omega$ as $n\to\infty$, it follows from the continuity of $f(x,\cdot)$ that $f(\cdot,u_n(\cdot))\to f(\cdot,\bar u(\cdot))$ a.e. in $\Omega$ as $n\to\infty$. Secondly, as $u_n,\bar{u}\in L^\infty(\Omega)$ for every $n\in\NN$, it follows from \eqref{f-L2} that $f(\cdot,u_n(\cdot)), f(\cdot,\bar u(\cdot))\in L^2(\Omega)$ for every $n\in\NN$. Thirdly, it follows from \eqref{u-boun} that  there is a constant $M>0$ (independent of $n$) such that $|u_n(x)|\le M$ for a.e. $x\in\Omega$ . Since $f(x,\cdot)$ is strictly  increasing and $f(x,0)=0$, for a.e. $x\in\Omega$, we have that $|f(x,u_n(x))|\le f(x,M)$ and $f(\cdot,M)\in L^2(\Omega)$ by \eqref{f-L2}. Therefore, using the Lebesgue Dominated Convergence Theorem, we get that $f(\cdot,u_n(\cdot))\to f(\cdot,\bar u(\cdot))$ in $L^2(\Omega)$ as $n\to\infty$.
Thus, taking the limit as $n\to\infty$
in both sides of the following identity 
\[
 \int_{\Omega}(L_D)^{\frac s2}u_n (L_D)^{\frac s2}v\;dx +\int_{\Omega}f(x,u_n)v\;dx
     =\int_{\Omega}z_nv\;dx, \quad \forall \;v \in \cV_0 ,
\]
and using all the above convergences, we can conclude that $\bar{u}$ is the weak solution of \eqref{ellip-pro} corresponding to the right hand side
$\bar{z}$. It remains to show that $\bar{z}$ is the optimal control. 
This is due to the lower semicontinuity of $J_1$, i.e., since $S(z_n)\to S(\bar{z})$ in $L^2(\Omega)$ as $n\to\infty$, we have that 
%and $J_1$ is weakly lower-semicontinuous (cf.~(i)), we have that 
%the composition $J_1\circ S : L^2(\Omega) \rightarrow [0,+\infty)$ is weakly lower-semicontinuous, i.e., 
$\liminf_{n\to\infty} J_1(S(z_n)) \ge J_1(S(\bar{z}))$. Since $J_2$ is  convex, proper and lower semicontinuous  (cf.~(ii)), it follows that it is weakly lower semicontinuous (see e.g. \cite[Theorem~3.3.3]{ABM}), i.e., $\liminf_{n\to\infty} J_2(z_n) \ge 
J_2(\bar{z})$. It then follows that 
\[
\inf_{z\in Z_{ad}} \mathcal{J}(z)
= \liminf_{n\rightarrow\infty} J(u_n,z_n) 
\ge J(S(\bar{z}),\bar{z}) \ge \inf_{z\in Z_{ad}} \mathcal{J}(z),
\]
and this completes the proof.
\end{proof}

\begin{remark}\label{new-rem}
{\rm We notice the following facts.
\begin{enumerate}
\item First, we mention that the nonlinearity $f$ given in 
Example \ref{exam} also satisfies  \eqref{f-L2}.   
\item Second, we  notice that all the results proved in this section also hold for  
$\widetilde S:L^{\widetilde{p}}\to H_0^{s}(\bOm)\cap L^\infty(\Omega)$, that is, the solution operator to  \eqref{ellip-int-Frac}.
\end{enumerate}
}
\end{remark}

\begin{remark}[\bf Nonsmooth cost functionals]
\label{rem:non_smooth}
{\em 
Notice that the condition $(i)$ in Theorem~\ref{thm:exist} already allows nonsmooth $J_1$ such as
$J_1(u):=\|u-u_d\|_{L^1(\Omega)}$ which is non-smooth but Lipschitz continuous, where $u_d \in L^1(\Omega)$ is given. 
The proof in Theorem~\ref{thm:exist} extends in a straightforward manner to other nonsmooth
control regularizations such as BV-regularization, i.e., when  
$J_2(z) := \int_\Omega |\nabla z|$ or when 
$J(u,z) := J_1(u) + J_2(z) + \int_\Omega |z| \;dx$ (with $J_1$ and $J_2$ as in Theorem~\ref{thm:exist}) 
by following \cite[Theorem~3.4]{antil2016optimal}, see also \cite{antil2018b}}.
\end{remark}

%\begin{remark}\label{rem:int_oc}
%{\rm We also notice that all the results proved in this section also hold for the map 
%$\widetilde S:L^\infty(\Omega)\to\widetilde\cV_0\cap L^\infty(\Omega)$
%$\widetilde S:L^{\widetilde{p}}\to H_0^{s}(\bOm)\cap L^\infty(\Omega)$}, that is, the solution operator to  \eqref{ellip-int-Frac}.
%\end{remark}

\section{Optimal control problem: first and second order optimality conditions}
\label{s:oc}

Before we investigate  more regularity of the control-to-state map, we make the following further regularity assumptions on the nonlinearity $f$.

\begin{assumption}\label{assum-diff}
We assume the following.
\begin{enumerate}[(i)]
\item The function $f(x,\cdot)$ is $k$-times, with $k=1,2$, continuously differentiable for a.e. $x\in\Omega$.

\item 
For all $M>0$ there 
exists a constant $L_{M}>0$ such that $f$ satisfies \eqref{Lip-cond} and
\[
  \left| \frac{\partial^k f}{\partial u^k} (x,\xi) - \frac{\partial^k f}{\partial u^k} (x,\eta)\right|
   \le L_{M}|\xi-\eta| , 
\]
for a.e. $x\in \Omega$ and all $\xi,\eta \in \RR$ with $|\xi|,|\eta| \le M$. 

\item $\displaystyle\frac{\partial f}{\partial u}(\cdot,0) \in L^\infty(\Omega)$. 

\item $\displaystyle\frac{\partial^2f}{\partial u^2}(\cdot,0) \in L^\infty(\Om)$.
%\item $\displaystyle\frac{\partial^2f}{\partial u^2}(\cdot,u(\cdot)) \in L^\infty(\Om)$ whenever $u(\cdot) \in 
%L^\infty(\Om)$.
\end{enumerate}
\end{assumption}

\begin{remark}\rm 
We mention the following facts.
\begin{enumerate}
\item For notational convenience, we will write $f_u$ and $f_{uu}$ instead of $\frac{\partial f}{\partial u}$ and $\frac{\partial^2 f}{\partial u^2}$, respectively. 
\item If $f$ satisfies Assumption \ref{assum-diff}(i)-(ii) with $k=1$ and $u\in L^\infty(\Om)$, then (iii) implies that there is a constant $C>0$ such that
\begin{align}\label{eq:fubd}
\| f_u(\cdot,u)\|_{L^\infty(\Omega)} &\le \| f_u(\cdot,u)-f_u(\cdot,0)\|_{L^\infty(\Omega)}
 + \| f_u(\cdot,0)\|_{L^\infty(\Omega)} \nonumber \\
 &\le C\left( \|u\|_{L^\infty(\Omega)} + \| f_u(\cdot,0)\|_{L^\infty(\Omega)} \right)< \infty.
\end{align}
\item Similarly, if $f$ satisfies Assumption~\ref{assum-diff}(i)-(ii) with $k=2$ and $u\in L^\infty(\Om)$, then (iv) implies that there is a constant $C>0$ such that
\begin{align}\label{eq:fuubd}
\| f_{uu}(\cdot,u)\|_{L^\infty(\Omega)} \le C \left(\|u\|_{L^\infty(\Omega)} + \| f_{uu}(\cdot,0)\|_{L^\infty(\Omega)}\right) < \infty.
\end{align}
\end{enumerate}
\end{remark}

Next we show that $S$ is twice continuously Fr\'echet differentiable. 

\begin{lemma}[\bf Twice Fr\'echet differentiability of $S$]
\label{lem:S2d}
Let Assumptions \ref{assum3} and \ref{assum-diff} hold. 
%\ref{assum-diff}(i), and \ref{assum-diff}(iv) hold
Then the mapping 
$S:L^{\widetilde{p}}(\Omega) \rightarrow \mathbb H^s(\Omega) \cap L^\infty(\Omega)$, where 
$\widetilde{p}$, is as in \eqref{eq:SL2Linf}, 
%$S:L^\infty(\Omega) \rightarrow \mathbb H^s(\Omega) \cap L^\infty(\Omega)$ 
is twice continuously Fr\'echet differentiable. Moreover, for all $z,\zeta \in 
L^{\widetilde{p}}(\Omega)$, 
$S'(z)\zeta = u_{\zeta} \in \mathbb H^s(\Omega) \cap L^{\infty}(\Om)$
is defined as the unique weak solution of 
\begin{equation}\label{eq:uv}
 (L_D)^s u_{\zeta} + f_u(\cdot,u)u_{\zeta} 
  = \zeta \quad \mbox{in } \Omega ,
\end{equation}
where $u=S(z)$. Furthermore, for every $z,\zeta_1,\zeta_2 \in L^{\widetilde{p}}(\Omega)$,
\begin{align*}
 S''(z)[\zeta_1,\zeta_2]:=(S''(z)\zeta_1)\zeta_2 = u_{\zeta_1,\zeta_2}\in
 \mathbb H^s(\Omega) \cap L^{\infty}(\Om), 
 \end{align*}
 is the unique weak solution of 
\begin{equation}\label{eq:uvv}
 (L_D)^s u_{\zeta_1,\zeta_2} +f_u(\cdot,u)u_{\zeta_1,\zeta_2} 
  = - f_{uu}(\cdot,u)u_{\zeta_1}u_{\zeta_2}  \quad \mbox{in } \Omega ,
\end{equation}
where $u=S(z)$ and $u_{\zeta_i}=S'(z)\zeta_i$, $i=1,2$.
\end{lemma}

\begin{proof}
The proof is based on the implicit function theorem. 
Let $\widetilde p$ be as in \eqref{eq:SL2Linf}.
We introduce the space 
\[
\cW = \Big\{u \in \mathbb H^s(\Omega)\cap L^\infty(\Omega), \;\; (L_D)^s u\in L^{\widetilde p}(\Omega) \Big\} 
\]
with norm 
\[
\|u\|_{\cW} = \|u\|_{\mathbb H^s(\Omega)\cap L^\infty(\Omega)} + \|(L_D)^su\|_{L^{\widetilde p}(\Omega)} .
\]
We next introduce the function
\[
\mathbb F : \cW \times L^{\widetilde{p}}(\Omega) \rightarrow L^{\widetilde p}(\Omega) , \quad \mathbb F(u,z) = (L_D)^s u + f(\cdot,u) - z ,
\]
with $(u,z) \in \cW \times L^{\widetilde{p}}(\Om)$.
Under Assumption~\ref{assum-diff}(i), $\mathbb F$ is of class $C^2$. Moreover, Assumptions \ref{assum3} and \ref{assum-diff}(i) imply that $f_u \ge 0$. Notice that since $u\in L^\infty(\Om)$, it follows from  \eqref{eq:fubd} that $f_u(\cdot,u)\in L^\infty(\Omega)$.
Then applying Theorem~\ref{theo-bound} 
to the operator 
 \[ 
  \mathbb F_u(u,z) = (L_D)^s + f_u(\cdot,u),
 \]
%with $f\equiv0$ 
we deduce that 
$
 \mathbb F_u(u,z) %= (L_D)^s + f_u(\cdot,u) 
$
is an isomorphism from $\cW$ to $L^{\widetilde p}(\Omega)$.  
Since $\mathbb F(u,z) = 0$ if and only if $u = S(z)$,
we can apply the implicit function theorem \cite[Theorem~2.7.2]{nirenberg1974topics} to deduce that $S$ is of class $C^2$ 
and fulfills 
$\mathbb F(S(z),z) = 0$. Therefore \eqref{eq:uv} follows easily. Moreover, \eqref{eq:uvv} follows after (in addition) using Assumption~\ref{assum-diff}(iv). The proof is finished.
\end{proof}

Throughout the remainder of the paper we restrict ourselves to the case where 
\begin{equation}\label{eq:J1J2}
 J_1(u) = \frac12\|u-u_d\|_{L^2(\Omega)}^2 \quad  \mbox{ and }\quad
 J_2(z) = \frac{\mu}{2} \|z\|^2_{L^2(\Omega)} . 
\end{equation}
The given function $u_d \in L^2(\Omega)$ and $\mu>0$ is the cost of the control. 
We further remark that these results can be directly extended to a more general
setting as described in the monograph \cite{MR2583281}.

Next, we introduce the adjoint state $\phi\in \mathbb H^s(\Omega)$ as the unique weak solution of
the adjoint equation
\begin{equation}\label{eq:adj}
%\begin{aligned}
(L_D)^s \phi+f_u(\cdot,u) \phi = u-u_d \quad \mbox{in } \Omega ,\\
%p &=0 \quad  \mbox{on } \partial\Omega 
%\end{aligned}
\end{equation}
where $u \in L^\infty(\Omega)$ is given. Using Assumptions \ref{assum3} and \ref{assum-diff}(i) we have that $f_u(x,u(x)) \ge 0$ for a.e. $x\in\Omega$. 
%Moreover, 
%\begin{align}\label{eq:fubd}
%\| f_u(\cdot,u)\|_{L^\infty(\Omega)} &\le \| f_u(\cdot,u)-f_u(\cdot,0)\|_{L^\infty(\Omega)}
% + \| f_u(\cdot,0)\|_{L^\infty(\Omega)} \nonumber \\
% &\le C \|u\|_{L^\infty(\Omega)} + \| f_u(\cdot,0)\|_{L^\infty(\Omega)} < \infty,
%\end{align}
%where in the last step we have used Assumption~\ref{assum-diff}(ii)-(iii).

\begin{lemma}[\bf Existence of solutions to the adjoint equation]
\label{adjoint:exist}
Let Assumptions \ref{assum3} and \ref{assum-diff}(i)--(iii) hold for $k=1$. 
Let $u\in L^\infty(\Omega)$ and $u_d \in L^2(\Omega)$ be given. Then there exists a unique $\phi \in\mathbb H^s(\Omega)$ 
weak solution to \eqref{eq:adj}. In addition, $\phi \in \mathbb{H}^{2s}(\Omega)$.
Finally, if $u_d$ in $L^p(\Om)$ where $p$ is as in \eqref{cond-p}, then $\phi \in L^\infty(\Om)$.
\end{lemma}
\begin{proof}
Since $f_u(\cdot,u) \in L^\infty(\Omega)$ and $f_u(\cdot,u) \ge 0$, the existence and uniqueness follow by using Assumption~\ref{assum-diff}(iii) and Proposition~\ref{prop-exis}. Since (notice that \eqref{eq:adj} is a linear equation in $\phi$)
\[
 \phi_k = \lambda_k^{-s} \int_\Omega (u-u_d - f_u(\cdot,u)\phi)\varphi_k\;dx,
\] 
then using the definition of the $\mathbb{H}^{2s}$-norm (see \eqref{norm-2}) we can deduce that 
\[
\|\phi\|_{\mathbb{H}^{2s}(\Omega)}  = \|u-u_d - f_u(\cdot,u)\phi\|_{L^2(\Omega)}.
\]
The $L^\infty(\Om)$-regularity of $\phi$ follows from 
Theorem~\ref{theo-bound}.
The proof is finished.
\end{proof}

\begin{lemma}[\bf $\mathcal{J}$ is twice Fr\'echet differentiable]
\label{lem:JFdiff}
Let the assumptions of Lemma~\ref{lem:S2d} and Lemma~\ref{adjoint:exist} hold. 
%and Assumption~\ref{assum-diff}(iv)} hold. 
Then the functional $\mathcal{J}:L^{\widetilde{p}}(\Omega)\rightarrow\mathbb{R}$ is twice continuously 
Fr\'echet differentiable. Moreover for every $z,\zeta,\zeta_1,\zeta_2 \in L^{\widetilde{p}}(\Omega)$ 
there holds 
\[
 \mathcal{J}'(z)\zeta = \int_\Omega (\phi + \mu z)\zeta\;dx ,
\]
and 
\begin{align*}
 \mathcal{J}''(z)[\zeta_1,\zeta_2] 
 %:= (\mathcal{J}''(z)\zeta_1)\zeta_2
 = \int_\Omega \Big(S'(z)\zeta_1 S'(z)\zeta_2  
  -\phi f_{uu}(x,S(z))S'(z)\zeta_1 S'(z)\zeta_2\Big)dx
  + \mu \int_\Omega \zeta_1 \zeta_2\;dx . 
\end{align*}
\end{lemma}

\begin{proof}
The proof is based on the chain rule,  the results from Lemma~\ref{lem:S2d} together with the fact that
\[
 \int_\Omega (S(z)-u_d) S'(z) v \;dx= \int_\Omega \phi v \;dx
\]
and 
\[
 \int_\Omega (S(z)-u_d) S''(z)[\zeta_1,\zeta_2] \;dx
 = - \int_\Omega \phi f_{uu}(x,S(z))S'(z)\zeta_1 S'(z)\zeta_2\;dx,
\]
which can be deduced from the weak formulations of \eqref{eq:uv}, \eqref{eq:uvv} and 
\eqref{eq:adj}.
\end{proof}

Since $\mathcal{J}$ is non-convex, in general due to the semilinear state equation,
we cannot expect a unique solution to the optimal control problem. We introduce locally 
optimal solutions: 
$\bar{z} \in \Zad$ is locally optimal or local solution to \eqref{RPS} if there exists an 
$\varepsilon>0$ such that 

\[
 \mathcal{J}(\bar{z}) \le \mathcal{J}(z) \quad \forall z \in \Zad \cap B_\varepsilon(\bar z),
\]
where the $L^{\widetilde{p}}$-ball $B_\varepsilon(\bar z)$ centered at $\bar{z}$ with radius $\varepsilon$
is defined by 
\[
 B_{\varepsilon}(\bar z) 
  := \Big\{z \in L^{\widetilde{p}}(\Omega), \;\; \|z-\bar{z}\|_{L^{\widetilde{p}}(\Omega)} \le \varepsilon \Big\} . 
\]

\begin{theorem}[\bf First order necessary conditions]
\label{thm:fooc}
For every local solution $\bar{z}$ of \eqref{RPS}, there exist a unique
optimal state $\bar{u}=S(\bar z)$ and an optimal adjoint state $\bar{\phi}$ such that
\begin{equation}\label{FoC}
\int_{\Omega}(\bar{\phi}+\mu\bar{z})(z-\bar{z})\;dx\ge 0 \quad \forall z \in \Zad,
\end{equation}
which is equivalent to 
\begin{equation}\label{proj}
\bar{z}(x) = \Pi_{[z_a(x),z_b(x)]}\left(-\frac{1}{\mu}\bar{\phi}(x)\right) 
\quad \mbox{ for a.e. } x \in \Omega . 
\end{equation}
Here, 
$\Pi_{[z_a(x),z_b(x)]}(w(x)) = \min\Big\{z_b(x),\max\{z_a(x),w(x)\}\Big\}$.
\end{theorem}
\begin{proof}
The proof of \eqref{FoC} is standard, see \cite[Lemma~4.18]{MR2583281}.  
For the equivalence between \eqref{FoC} and \eqref{proj} we refer to \cite[page~217]{MR2583281}. 
\end{proof}

 \begin{remark}\label{rem:strong}
  A necessary and sufficient condition for \eqref{FoC} to hold is that for 
  a.e. $x \in \Om$,
  \begin{equation}\label{eq:a}
  \begin{cases}
   \bar{z}(x) =
              z_a(x)\;\;& \mbox{ if }\; \bar{\phi}(x)+\mu\bar{z}(x) > 0, \\
            \bar{z}(x)   \in [z_a(x),z_b(x)] \; & \mbox{ if } \;\bar{\phi}(x)+\mu\bar{z}(x) = 0, \\
           \bar{z}(x) =   z_b(x)  \;& \mbox{ if } \;\bar{\phi}(x)+\mu\bar{z}(x) < 0 ,
             \end{cases}
  \end{equation}
  or equivalently the following pointwise inequality in $\R$:
  \[
   (\bar\phi(x) + \mu \bar{z}(x))(z - \bar{z}(x)) \ge 0 \quad \forall z \in [z_a(x),z_b(x)] 
    \quad \mbox{ and for a.e. } x \in \Om . 
  \]
  We refer to \cite[Lemma~{2.26}]{MR2583281} for a proof. We notice from \eqref{eq:a} 
  that $|\bar\phi(x) + \mu \bar{z}(x)| > 0$ implies $\bar{z} = z_a$ or $\bar{z} = z_b$. 
  The set of all $x \in \Om$ where $|\bar\phi(x) + \mu \bar{z}(x)| > 0$ will be called later,
  a strongly active set. 
 \end{remark}

\begin{remark}[\bf Nonsmooth cost functionals]
{\em  We let $J_1$ be as in \eqref{eq:J1J2}. 
The first order optimality conditions when $J_2(z) := \int_\Omega |\nabla z|$ are technical and 
are part of our future project (cf.~\cite{casas2017analysis} for the case of the classical Laplacian). 
On the other hand in case $J(u,z) = J_1(u) + J_2(z) + \nu\|z\|_{L^1(\Omega)}$ with $J_1, J_2$ as in \eqref{eq:J1J2} and constants $z_a, z_b$ fulfilling $z_a < 0 < z_b$, the first order optimality 
conditions are a modification of \eqref{proj} by using the characterization of the subdifferential of 
the $L^1(\Omega)$-norm (cf. \cite[Corollary~3.2]{ECasas_RHerzog_GWachsmuth_2012a} and 
\cite{FHClarke_1990a} for details). In particular, for a.e. $x\in\Omega$, we obtain the following:
\begin{enumerate}
\item $\bar{z}(x) = \Pi_{[z_a,z_b]}\left(-\frac{1}{\mu}(\bar{\phi}(x)+\nu\bar{\zeta}(x))\right)$;
\item $\bar{z}(x) = 0$ if and only if $|\bar{\phi}(x)|\le \nu$;
\item $\bar{\zeta}(x) = \Pi_{[-1,1]}\left(-\frac{1}{\nu}\bar{\phi}(x)\right)$. 
\end{enumerate}
}
\end{remark}

Before we state the second order necessary and sufficient conditions we 
first introduce the set of strongly active constraints (or strongly active set) $A_\tau(\bar{z})$ (see also
Remark~\ref{rem:strong}).
For $\tau \ge 0$ an arbitrary but fixed parameter, we let
 \[
  A_\tau(\bar{z}) := \{ x \in \Om \ : \ |\bar{\phi}(x)+\mu\bar{z}(x)| > \tau \}  .
 \]
We also define the $\tau$-critical cone (cf. \cite{MR1316261}) associated to a control $\bar{z}$ as
\begin{equation}\label{cone1}
 C_\tau(\bar{z}) := \Big\{v \in L^{\widetilde{p}}(\Omega), \;\: v \mbox{ fulfills } \eqref{cone2} \Big\} ,
\end{equation}
where for a.e. $x\in\Omega$,
 \begin{equation}\label{cone2}
 v(x) \left\{\begin{array}{ll}
                \ge 0, & \mbox{if }  \bar{z}(x) = z_a(x) , \\                   
                \le 0, & \mbox{if }  \bar{z}(x) = z_b(x) , \\                  
                0 ,    & \mbox{if } x \in A_\tau(\bar{z}) .
               \end{array}
        \right. 
\end{equation}
%We remark that if $x \not\in (\Om\setminus A_\tau(\bar{z}))$, then 
%$\bar{\phi}(x) + \mu\bar{z}(x) = 0$ which is due to the definition of 
%$A_\tau(\bar{z})$ and \eqref{eq:a}. 

%\MW{Why did you delete it? It is better to give the correct statement.}
%Now we are ready to state the second order necessary conditions.

 \begin{proposition}[\bf Second order necessary conditions]
  Let $\bar{z}\in\Zad$ be a locally optimal control. Then $\mathcal{J}''(\bar{z})[z,z] \ge 0$
  for all $z \in C_0(\bar{z})$.
 \end{proposition}
 \begin{proof}
  The proof is identical to the classical case $s=1$ and is contained in 
  \cite[Theorem~4.27]{MR2583281}.
 \end{proof}

\begin{theorem}[\bf Quadratic growth condition]
\label{ssc}
Let $\bar{z}\in\Zad$ be a control satisfying the first order optimality conditions 
\eqref{FoC}. Assume that there exist two constants $\tau>0$ and $\delta>0$ such that
\begin{equation}\label{eq:ass_ssc}
 \mathcal{J}''(\bar{z})[z,z] \ge \delta \|z\|_{L^2(\Omega)}^2 
 \quad \forall z \in C_\tau(\bar z) .
\end{equation}
Then there are two constants $\varrho>0$ and $\beta>0$ such that 
\begin{equation}\label{eq:quad_grow}
 \mathcal{J}(z) \ge \mathcal{J}(\bar z) + \beta \|z-\bar z\|_{L^2(\Omega)}^2 
 \quad \forall z \in \Zad \cap B_{\varrho}(\bar z) . 
\end{equation}
\end{theorem}
Notice that \eqref{eq:ass_ssc} can be easily checked in several cases, for instance, 
when the state equation is linear and, as a result the underlying optimal control problem is strictly convex. Nevertheless, in certain cases of non-convex optimal control problems it is also possible to prove \eqref{eq:ass_ssc}, see 
for instance \cite{MR3256797}. Before proving Theorem~\ref{ssc} we need the following auxiliary result which shall help us to deal with the 2-norm discrepancy. Notice that the 2-norm discrepancy 
only occurs when $\widetilde{p} = +\infty$.

\begin{lemma}
\label{Lem:ass_ssc}
Assume that Assumption~\ref{assum-diff} holds and let $\cJ:L^{\widetilde{p}}(\Omega)\rightarrow\RR$. Then for each $M>0$ there is a constant $L(M)>0$ such that
\begin{equation}\label{E-123}
 |\cJ''(z+h)[z_1,z_2]-\cJ''(z)[z_1,z_2]| 
\le L(M) \|h\|_{L^{\widetilde{p}}(\Omega)} \|z_1\|_{L^2(\Omega)} \|z_2\|_{L^2(\Omega)} ,
\end{equation}
for all $z,h,z_1,z_2 \in L^{\widetilde{p}}(\Omega)$ satisfying $\max\{\|z\|_{L^{\widetilde{p}}(\Omega)},\|h\|_{L^{\widetilde{p}}(\Omega)}\} \le M$. 
\end{lemma}

\begin{proof}
We begin by setting $u = S(z)$, $u_h = S(z+h)$ with the corresponding adjoint state $\phi$ and $\phi_h$, respectively. Moreover, let $u_i = S'(z)z_i$ and $u_{i,h} = S'(z+h)z_i$ for $i = 1,2$. Using Lemma~\ref{lem:JFdiff} we have 
\begin{align*}
&\cJ''(z+h)[z_1,z_2] - \cJ''(z)[z_1,z_2] \\
& \ = \int_\Omega \left(u_{1,h} u_{2,h} - u_1 u_2 \right)\;dx 
-\int_\Omega\phi_h f_{uu}(x,u_h) u_{1,h} u_{2,h} \;dx + \int_\Omega \phi f_{uu}(x,u) u_1 u_2 \;dx\\
& \ =\int_\Omega \left(u_{1,h} u_{2,h} - u_1 u_2 \right)\;dx 
-\int_\Omega\phi_h \left(f_{uu}(x,u_h) u_{1,h} u_{2,h} - f_{uu}(x,u) u_1 u_2\right)\;dx \\
& \quad + \int_\Omega (\phi-\phi_h) f_{uu}(x,u) u_1 u_2 \;dx.
\end{align*}
Therefore 

\begin{align*}
&\left| \cJ''(z+h)[z_1,z_2] - \cJ''(z)[z_1,z_2] \right| \\
&\ \le \|u_{1,h} u_{2,h} - u_1 u_2 \|_{L^1(\Omega)} 
 + \|\phi-\phi_h\|_{L^\infty(\Omega)} \|f_{uu}(\cdot,u) u_1 u_2\|_{L^1(\Omega)}  \\
&\quad+ \|\phi_h\|_{L^\infty(\Omega)} \|f_{uu}(\cdot,u_h) u_{1,h} u_{2,h} - f_{uu}(\cdot,u) u_1 u_2 \|_{L^1(\Omega)} 
 = (I) +( II) + (III) .
\end{align*}
We shall estimate $(I)$, $(II)$ and $(III)$ separately. We have
 \begin{align}\label{E-I}
 ( I) \le \|u_{1,h}\|_{L^2(\Om)}\|u_{2,h}-u_2\|_{L^2(\Om)} + \|u_{1,h}-u_1\|_{L^2(\Om)}\|u_2\|_{L^2(\Om)} .
 \end{align}
From \eqref{eq:uv} we recall that $u_i$ and $u_{i,h}$ solve the linear equations
 \begin{equation}\label{lem:59_1}
  (L_D)^s u_{i} + f_u(\cdot,u)u_{i} 
   = z_i \quad \mbox{in } \Omega ,\quad \quad 
  (L_D)^s u_{i,h} + f_u(\cdot,u_h)u_{i,h} 
   = z_i \quad \mbox{in } \Omega.
 \end{equation}
Using the properties of these linear equations and the fact that $f_u \ge 0$, we get that 
$\|u_{i}\|_{L^2(\Om)} \le C \|z_i\|_{L^2(\Om)}$ and 
$\|u_{i,h}\|_{L^2(\Om)} \le C \|z_i\|_{L^2(\Om)}$ for some constant $C>0$. It then remains to estimate $\|u_{i,h}-u_i\|_{L^2(\Om)}$ for $i=1,2$. Subtracting the two equations in \eqref{lem:59_1} and rearranging the terms we 
obtain that 
 \begin{equation}\label{lem:59_2}
  (L_D)^s (u_{i,h}-u_i) + f_u(\cdot,u)(u_{i,h}-u_i)
   = -\left(f_u(\cdot,u_h) - f_u(\cdot,u) \right)u_{i,h} \quad \mbox{in } \Omega.
 \end{equation}
We can estimate the $L^2$-norm of the right hand side as 
 \[
  \|\left(f_u(\cdot,u_h) - f_u(\cdot,u) \right)u_{i,h}\|_{L^2(\Om)} 
  \le \|f_u(\cdot,u_h) - f_u(\cdot,u)\|_{L^\infty(\Om)} \|u_{i,h}\|_{L^2(\Om)}
  \le \|h\|_{L^{\widetilde{p}}(\Om)} \|z_i\|_{L^2(\Om)},
 \] 
where in the last step we have used the Lipschitz continuity of $f_u$ (Assumption~\ref{assum-diff}) and $S$ (cf. Lemma~\ref{lem:S_Lip} and Remark~\ref{rem:pvstp}),  and the aforementioned estimate for 
$\|u_{i,h}\|_{L^2(\Om)}$. Combining
this estimate with the well-posedness of \eqref{lem:59_2} (cf.~\eqref{eq:uv}) we obtain
(for $i=1,2$) that there is a constant $C>0$ such that
 \[
  \|u_{i,h}-u_i\|_{L^2(\Om)} \le C \|h\|_{L^{\widetilde{p}}(\Om)} \|z_i\|_{L^2(\Om)} . 
 \]
Substituting this in the above estimate \eqref{E-I} of $(I)$, we get that 
 \begin{equation}\label{E-I-1}
 ( I) \le C\|z_1\|_{L^2(\Om)} \|h\|_{L^{\widetilde{p}}(\Om)} \|z_2\|_{L^2(\Om)} . 
 \end{equation} 
Next we estimate $(II)$. Subtracting the adjoint equations (cf.~\eqref{eq:adj}) for $\phi$ and $\phi_h$ we
obtain that 
 \[
  L^s_D (\phi_h-\phi) + f_u(x,u)(\phi_h-\phi) 
   = u_h - u + \left( f_u(x,u) - f_u(x,u_h) \right)\phi_h .
 \]
From the well-posedness of this equation and the $L^\infty$-estimate (under the assumptions of Lemma~\ref{adjoint:exist}),
we get that there is a constant $C>0$ such that 
 \[
  \|\phi_h-\phi\|_{L^\infty(\Om)} 
   \le \|u_h-u\|_{L^\infty(\Om)} + \| f_u(\cdot,u) - f_u(\cdot,u_h) \|_{L^\infty(\Om)} 
           \|\phi_h\|_{L^\infty(\Om)} 
   \le C \|h\|_{L^{\widetilde{p}}(\Om)}   , 
 \]
where in the last step we have used the following: (a) the Lipschitz continuity of the control to state map
\eqref{Lip} (cf.~Remark~\ref{rem:pvstp}), i.e., $\|u_h-u\|_{L^\infty(\Om)} \le C\|h\|_{L^{\widetilde{p}}(\Om)}$ which further implies
that $\|u_h\|_{L^\infty(\Omega)}$ is uniformly bounded (cf. Lemma~\ref{lem:nodisp}); (b) 
the uniform boundedness of  $\phi_h$ which is due to the adjoint equation for $\phi_h$ (since $u_h$ is uniformly bounded); (c) the Lipschitz continuity of $f_u(x,\cdot)$. To complete the estimate for $(II)$, we need to estimate 
$\|f_{uu}(\cdot,u)u_1u_2\|_{L^1(\Om)}$. We notice that 
 \[
  \|f_{uu}(\cdot,u)u_1u_2\|_{L^1(\Om)} \le C \|f_{uu}(\cdot,u)\|_{L^\infty(\Om)} 
     \|u_1\|_{L^2(\Om)} \|u_2\|_{L^2(\Om)}
   \le C \|z_1\|_{L^2(\Om)} \|z_2\|_{L^2(\Om)} , 
 \]
where we have used the assumption that $f_{uu}$ is bounded when $u$ is bounded (cf.~\eqref{eq:fuubd} and Assumption~\ref{assum-diff}(i),(ii), and (iv)).  
As a result, we get that
 \begin{equation}\label{E-II}
 ( II) \le C \|h\|_{L^{\widetilde{p}}(\Om)}\|z_1\|_{L^2(\Om)} \|z_2\|_{L^2(\Om)} . 
 \end{equation}
Next we estimate $(III)$. We first recall that $\|\phi_h\|_{L^\infty(\Om)}$ is uniformly 
bounded with respect to $h$. Furthermore
 \begin{align*}
  &\|f_{uu}(\cdot,u_h) u_{1,h} u_{2,h} - f_{uu}(\cdot,u) u_1 u_2 \|_{L^1(\Omega)} \\
  &\quad\le \| \left( f_{uu}(\cdot,u_h) - f_{uu}(\cdot,u) \right) u_{1,h} u_{2,h} \|_{L^1(\Om)}
        + \|f_{uu}(\cdot,u) \left( u_{1,h} u_{2,h} - u_1 u_2 \right) \|_{L^1(\Om)} \\
  &\quad\le \|f_{uu}(\cdot,u_h) - f_{uu}(\cdot,u)\|_{L^\infty(\Om)} \|u_{1,h}\|_{L^2(\Om)} \|u_{2,h}\|_{L^2(\Om)}
   + \|f_{uu}(\cdot,u)\|_{L^\infty(\Om)}\|u_{1,h} u_{2,h} - u_1 u_2 \|_{L^1(\Om)} \\
  &\quad \le C \|h\|_{L^{\widetilde{p}}(\Om)} \|z_1\|_{L^2(\Om)} \|z_2\|_{L^2(\Om)}  ,
 \end{align*}
where in the last step we have used the Lipschitz continuity of $f_{uu}$ (Assumption~\ref{assum-diff}) (ii)) and the aforementioned estimates for
$\|u_{i,h}\|_{L^2(\Om)}$, $i=1,2$, 
%the control to state map $S$ to estimate $\|u_{i,h}\|_{L^2(\Om)}$, 
the aforementioned boundedness of $f_{uu}(\cdot,u)$ and the estimate for 
$\|u_{1,h} u_{2,h} - u_1 u_2 \|_{L^1(\Om)}$ as in the case of $(I)$. We have shown that
\begin{equation}\label{E-III}
(III)\le C \|h\|_{L^{\widetilde{p}}(\Om)}\|z_1\|_{L^2(\Om)} \|z_2\|_{L^2(\Om)} .
\end{equation}
Finally \eqref{E-123} follows from \eqref{E-I-1}, \eqref{E-II} and \eqref{E-III}. The proof is finished.
%The result then follows by estimating the terms on the right hand side. See for instance 
%\cite[Lemma~4.26]{MR2583281}.
\end{proof}

Now we are ready to give the proof of Theorem~\ref{ssc} which is inspired from \cite[Chapter~5]{MR2583281}. We emphasize that the proof can also be deduced from the general theory given in \cite{MR2902693}. In fact we shall apply the abstract result from \cite{MR2902693} for the case $\widetilde{p} = 2$.

\begin{proof}[\bf Proof of Theorem~\ref{ssc}]
Let $z \in Z_{ad}$ with $\|z-\bar{z}\|_{L^{\widetilde{p}}(\Om)} \le \varrho$. We set 
$h:= z - \bar{z}$. Applying Taylor's theorem we have that 
 \begin{align*}
  \cJ(z) &= \cJ(\bar{z}) + \cJ'(\bar{z})h  
      + \frac12 \cJ''(\bar{z}+\theta h)h^2  \\
         &= \cJ(\bar{z}) + \cJ'(\bar{z})h + \frac12 \cJ''(\bar{z})h^2
      + \frac12 \left( \cJ''(\bar{z}+\theta h) - \cJ''(\bar{z}) \right) h^2 
 \end{align*}
with $\theta = \theta(x) \in (0,1)$. Using Lemma~\ref{Lem:ass_ssc} we have that there is a constant $L>0$ such that
 \begin{align*}
  \cJ(z) &\ge \cJ(\bar{z}) + \cJ'(\bar{z})h + \frac12 \cJ''(\bar{z})h^2
          - L \|h\|_{L^{\widetilde{p}}(\Om)} \|h\|^2_{L^2(\Om)} . 
 \end{align*}
Since $\cJ'(\bar{z})h = \int_\Om (\mu \bar{z} + \bar{\phi}) h \;dx$ and also
$\|h\|_{L^{\widetilde{p}}(\Om)} \le \varrho$, we get that
 \begin{align*}
  \cJ(z) &\ge \cJ(\bar{z}) + \int_{A_\tau} (\mu \bar{z} + \bar{\phi}) h \;dx
  + \int_{\Om\setminus A_\tau} (\mu \bar{z} + \bar{\phi}) h \;dx
  + \frac12 \cJ''(\bar{z})h^2
          - \varrho L \|h\|^2_{L^2(\Om)} \\
         &\geq \cJ(\bar{z}) + \int_{A_\tau} (\mu \bar{z} + \bar{\phi}) h \;dx
  + \frac12 \cJ''(\bar{z})h^2
          - \varrho L \|h\|^2_{L^2(\Om)} ,
 \end{align*}
%because $\mu \bar{z} + \bar{\phi} = 0$ on $\Om\setminus A_\tau$ (due to the 
%definition of $A_\tau$ and \eqref{eq:a}). 
where we have used Remark~\ref{rem:strong}.
From the first-order necessary conditions we have that $\cJ'(\bar{z})h \ge 0$,
and as a result 
 \[
  \cJ(z) \ge \cJ(\bar{z}) + \tau \int_{A_\tau} |h(x)|\;dx + 
   \frac12 \cJ''(\bar{z})h^2  
   - \varrho L  \|h\|^2_{L^2(\Om)} .
 \]
We next split $h$ into two parts, $h = h_0+h_1$ as 
 \[
  h_0(x) := \left\{
             \begin{array}{ll} 
              h(x) & \mbox{if } x \not\in A_\tau \\
              0    & \mbox{if } x \in A_\tau .
             \end{array}
            \right.
 \]
Notice that $h_0$ fulfills the sign conditions in $C_\tau(\bar{z})$. Since
$h_0 = 0$ on $A_\tau$, then $h_0 \in C_\tau(\bar{z})$. Thus
 \begin{equation}\label{eq:sop}
  \cJ(z) \ge \cJ(\bar{z}) + \tau \int_{A_\tau} |h(x)|\;dx 
   + \frac12 \cJ''(\bar{z})(h_0+h_1)^2  
      - \varrho L  \|h\|^2_{L^2(\Om)} . 
 \end{equation}
Next we estimate $\frac12 \cJ''(\bar{z})(h_0+h_1)^2$. Since $h_0 \in C_\tau(\bar{z})$, then
from \eqref{eq:ass_ssc} we have that 
 \[
  \frac12 \cJ''(\bar{z}) h_0^2 \ge \frac{\delta}{2} \|h_0\|^2_{L^2(\Om)} . 
 \]
Let $\widetilde{p}'$ be such that $\frac{1}{\widetilde{p}}+\frac{1}{\widetilde{p}'}=1$ with the convention that $\widetilde{p}'=1$ if $\widetilde{p}=+\infty$. By applying Young's inequality we see that (for a generic constant $C$)
 \begin{align*}
  |\cJ''(\bar{z})[h_0,h_1]| 
   &\le C\|h_0\|_{L^2(\Om)} \|h_1\|_{L^2(\Om)}
   \le \frac{\delta}{4} \|h_0\|^2_{L^2(\Om)}  + C \|h_1\|_{L^2(\Om)}^2 \\
   &\le \frac{\delta}{4} \|h_0\|^2_{L^2(\Om)} + 
           C \|h_1\|_{L^{\widetilde{p}'}(\Om)} \|h_1\|_{L^{\widetilde{p}}(\Om)} \\
   &\le \frac{\delta}{4} \|h_0\|^2_{L^2(\Om)} + 
           C_1 \varrho \|h_1\|_{L^{\widetilde{p}'}(\Om)} ,        
 \end{align*} 
where in the last step we have used that 
$\|h_1\|_{L^{\widetilde{p}}(\Om)} \le \varrho$. Similarly, we have that 
 \[
  \Big|\frac12 \cJ''(\bar{z}) h_1^2\Big| \le C \|h_1\|^2_{L^2(\Om)}
   \le C_2 \varrho \|h_1\|_{L^{\widetilde{p}'}(\Om)}.
 \]
After summing the above inequalities we obtain that 
 \begin{align*}
  \frac12 \cJ''(\bar{z}) (h_0+h_1)^2 &\ge \frac{\delta}{2} \|h_0\|^2_{L^2(\Om)}
   - \Big( \frac{\delta}{4}\|h_0\|^2_{L^2(\Om)} + (C_1+C_2)\varrho \|h_1\|_{L^{\widetilde{p}'}(\Om)}\Big) \\
   &\ge \frac{\delta}{4} \|h_0\|^2_{L^2(\Om)} - (C_1+C_2)\varrho \|h_1\|_{L^{\widetilde{p}'}(A_\tau)}   ,  
 \end{align*} 
%By choosing $\varrho > 0$ small enough so that $(C_1+C_2)\varrho \le \frac{\tau}2$ and 
where in the last step we have also used that $h_1 = 0$ on $\Om\setminus A_\tau$. 
% \begin{align*}
%  \frac12 \cJ''(\bar{z}) (h_0+h_1)^2 &\ge 
%    \frac{\delta}{4} \|h_0\|^2_{L^2(\Om)} - (C_1+C_2)\varrho \|h_1\|_{L^{\widetilde{p}'}(A_\tau)} . 
% \end{align*} 
Substituting this in \eqref{eq:sop} and rearranging the  terms we obtain that 
 \begin{equation}\label{eq:bbbbb}
  \cJ(z) \ge \cJ(\bar{z}) 
   + \tau \|h\|_{L^1(A_\tau)} - (C_1+C_2)\varrho \|h\|_{L^{\widetilde{p}'}(A_\tau)}
             + \frac{\delta}{4} \|h\|^2_{L^2(\Om\setminus A_\tau)}  
             - \varrho L \|h\|^2_{L^2(\Om)} ,
 \end{equation}
where we have used that $\|h_1\|_{L^{\widetilde{p}'}(A_\tau)} = \|h\|_{L^{\widetilde{p}'}(A_\tau)}$ and $\|h_0\|_{L^2(\Om)} = \|h\|_{L^2(\Om\setminus A_\tau)}$. 
Now we consider two cases.

{\bf Case 1. $\widetilde{p} = +\infty$.} Then
$\widetilde{p}' = 1$. Choosing $\varrho$ small enough in \eqref{eq:bbbbb} so that 
$(C_1+C_2)\varrho \le \frac{\tau}2$ we obtain that 
 \begin{align*}
  \cJ(z) \ge \cJ(\bar{z}) 
   + \frac{\tau}{2} \|h\|_{L^1(A_\tau)}
             + \frac{\delta}{4} \|h\|^2_{L^2(\Om\setminus A_\tau)}     
             - \varrho L \|h\|^2_{L^2(\Om)} \\
         \ge \cJ(\bar{z}) 
   + \frac{\tau}{2} \|h\|_{L^2(A_\tau)}^2
             + \frac{\delta}{4} \|h\|^2_{L^2(\Om\setminus A_\tau)}     
             - \varrho L \|h\|^2_{L^2(\Om)}  ,  
 \end{align*}
where we have assumed that $\varrho \le 1$ (without loss of generality) which implies that $h(x)^2 \le h(x)$ for a.e. $x\in\Omega$ and thus $\|h\|_{L^1(A_\tau)}\ge \|h\|_{L^2(A_\tau)}^2$. Thus, we obtain
 \begin{align*}
  \cJ(z)
         &\ge \cJ(\bar{z}) 
   + \min\Big\{ \frac{\tau}{2},\frac{\delta}{4}\Big\}  \|h\|_{L^2(\Om)}^2
             - \varrho L \|h\|^2_{L^2(\Om)}    \\
         &\ge \cJ(\bar{z})  + \frac12 \min\Big\{ \frac{\tau}{2},\frac{\delta}{4}\Big\} 
          \|h\|_{L^2(\Om)}^2    ,
 \end{align*}
where in the last step we have set $\varrho = \frac{1}{2L} \min\Big\{ \frac{\tau}{2},\frac{\delta}{4}\Big\}$ which can be made small enough since $L$ can be made large enough. We have shown \eqref{eq:quad_grow}.

{\bf Case 2. $\widetilde{p} = 2$.} 
Then $\widetilde{p}' = 2$. 
We notice that there is no $2$-norm discrepancy for us in this case, thus the situation is similar to \cite[Theorem 4.29]{pfefferer2014numerical}. 
We shall use the results in \cite[Theorem~2.3 and Theorem~2.7]{MR2902693}. 
%This requires checking a few assumptions as is done in Corollary 3.6 and 3.7 in \cite{MR2902693}:

    \begin{enumerate}[(i)]
        \item
%              Assumption (A1) from \cite{MR2902693} is immediately fulfilled in our 
%              case of quadratic cost, since 
              Firstly, we notice that 
              $\mathcal{J}: L^2(\Omega) \rightarrow  \mathbb{R}$ is of class $C^2$. 
              Moreover, $\mathcal{J}'(z)$ and $\mathcal{J}''(z)$ can be extended to linear 
              and bilinear forms, respectively, on $L^2(\Omega)$.

        \item 
%              Assumption (A2) from \cite{MR2902693}: We need to show that the Equations (2.3)-(2.5)
%              of \cite{MR2902693} holds. 
              %We can argue similarly to Proposition~3.4 \cite{MR2902693} where the standard 
              %$(s=1)$ case with Neumann boundary control problem has been
              %discussed.         
              %This amounts to showing the following: 
              Let $\{(z_k,\zeta_k)\}_{k} \in 
              Z_{ad} \times L^2(\Omega)$ be a sequence such that $\|z_k-z\|_{L^2(\Omega)} \rightarrow 0$
              and $\zeta_k \rightharpoonup \zeta$ weakly in $L^2(\Omega)$ as $k\to\infty$. We claim that 
              \begin{align}
                  &\mathcal{J}'(z) \zeta  = \lim_{k\rightarrow\infty} \mathcal{J}'(z_k)\zeta_k, 
                      \label{eq:5101} \\
                  &\mathcal{J}''(z) [\zeta,\zeta] 
                      \le \liminf_{k\rightarrow\infty} \mathcal{J}''(z_k)[\zeta_k,\zeta_k] ,
                              \label{eq:5102}\\
                  &\mbox{if } \zeta = 0, \mbox{ then} 
                          \quad C \liminf_{k\rightarrow\infty} \|\zeta_k\|^2_{L^2(\Omega)} 
                          \le \liminf_{k\rightarrow\infty} \mathcal{J}''(z_k) [\zeta_k,\zeta_k] , \label{eq:5103} 
              \end{align}      
              holds for some constant $C > 0$.

              %We will sketch the proof of each of these next. 
              Indeed, since $\{z_k\}_k$ converges to $z$
              in $L^2(\Omega)$ as $k\to\infty$, therefore using Lemma~\ref{lem:S2d} and the continuity of $S$, we get that $S(z_k) \rightarrow S(z)$
              in $\mathbb{H}^s(\Omega)\cap L^\infty(\Omega)$ as $k\to\infty$. Using this property in the adjoint equation \eqref{eq:adj},
              we can deduce that $\phi(z_k) \rightarrow \phi(z)$ in $\mathbb{H}^s(\Omega)\cap L^\infty(\Omega)$ as $k\to\infty$.  Next, since $(\phi(z_k)+\mu z_k)\to (\phi(z)+\mu z)$ in $L^2(\Omega)$ and $\zeta_k \rightharpoonup \zeta$ weakly in $L^2(\Omega)$, as $k\to\infty$, we have that $(\phi(z_k)+\mu z_k)\zeta_k\rightharpoonup (\phi(z)+\mu z)\zeta$ in $L^1(\Omega)$ as $k\to\infty$.
              Then from the expression of $\mathcal{J}'(z)\zeta$ in 
              Lemma~\ref{lem:JFdiff}, we can deduce that
              \[\lim_{k\to\infty}\mathcal{J}'(z_k)\zeta_k = \lim_{k\to\infty}\int_\Omega \left( \phi(z_k) + \mu z_k \right) \zeta_k\, dx
             = \int_\Omega \left( \phi(z) + \mu z \right) \zeta\, dx = \mathcal{J}'(z)\zeta,\]
               and we have shown
              \eqref{eq:5101}.

              %Next we will show \eqref{eq:5102}. 
              %\MW{\bf Remove the first stentence as it is obsolete?}
%              From \eqref{eq:uv} (recall that $S'(z_k)\zeta_k$ is the solution of \eqref{eq:uv} with right hand side $\zeta_k$) we have that (up to subsequences if needed)
%              $S'(z_k)\zeta_k \rightarrow S'(z)\zeta$  in $L^2(\Omega)$ as $k\to\infty$. Here we used the continuous and compact embedding $\mathbb{H}^{s}(\Omega) \hookrightarrow L^2(\Om)$. 
              From the expression of 
              $\mathcal{J}''$ in Lemma~\ref{lem:JFdiff}, we have that
              \begin{align}\label{EX-J}
              \mathcal{J}''(z_k)[\zeta_k,\zeta_k] 
                 = \int_\Omega\left( |S'(z_k)\zeta_k|^2  
                  -\phi(z_k) f_{uu}(x,S(z_k))|S'(z_k)\zeta_k|^2\right) \;dx
                  + \mu \int_\Omega \zeta_k^2\;dx.
              \end{align}
              Since $S'(z_k)\zeta_k$ is the unique weak solution of \eqref{eq:uv} with right-hand side $\zeta_k$, and  $\zeta_k \rightharpoonup \zeta$ weakly in $L^2(\Omega)$ as $k\to\infty$, hence the sequence $\{\zeta_k\}_{k\in\NN}$ is bounded, we have that $S'(z_k)\zeta_k$ is bounded in $\mathbb H^s(\Omega)$. Thus, after a subsequence if necessary, we have that $S'(z_k)\zeta_k$ converges weakly to $S'(z)\zeta$ in $\mathbb H^s(\Om)$ and strongly in $L^2(\Omega)$ (since the embedding $\mathbb H^s(\Omega)\hookrightarrow L^2(\Omega)$ is compact) as $k\to\infty$. Using similar arguments together with the properties of $f_{uu}$ and  $\phi$ we can deduce that    
              \[\lim_{k\to\infty}\int_{\Omega}\left(\phi(z_k) f_{uu}(x,S(z_k))|S'(z_k)\zeta_k|^2\right)\;dx=\int_{\Omega}\left( \phi(z) f_{uu}(x,S(z))|S'(z)\zeta|^2\right)\;dx.\]
             Now taking the limit as $k\to\infty$ of \eqref{EX-J} and using the above convergences, we get
              \begin{align*}
                 \liminf_{k\rightarrow\infty}\mathcal{J}''(z_k)[\zeta_k,\zeta_k] 
                 &= \lim_{k\rightarrow\infty}\int_\Omega\left( |S'(z_k)\zeta_k|^2  
                  -\phi(z_k) f_{uu}(x,S(z_k))S'(z_k)\zeta_k S'(z_k)\zeta_k\right) \;dx
                  + \mu \liminf_{k\rightarrow\infty} \int_\Omega \zeta_k^2\;dx \\
                 &\ge \int_\Omega\left( |S'(z)\zeta|^2  
                  -\phi(z) f_{uu}(x,S(z))S'(z)\zeta S'(z)\zeta \right)\;dx 
                  + \mu \int_\Omega \zeta^2\;dx = \mathcal{J}''(z)[\zeta,\zeta],
              \end{align*}            
              where in the last integral, we have used the weak lower semicontinuity
              of the $L^2(\Omega)$-norm.  We have shown \eqref{eq:5102}.

              It remains to show \eqref{eq:5103}. Since $\zeta = 0$, then all the integral 
              terms except the last one in the expression of $\mathcal{J}''(z_k)[\zeta_k,\zeta_k]$ vanish and 
              \eqref{eq:5103} follows with $C = \mu$.

        \item
              The relation between $C_\tau(\bar{z})$ in \eqref{cone1} and $C_{\bar{z}}$ in 
              \cite[ Theorem~2.3]{MR2902693}: In case $z_a,z_b$ are constants this 
              is discussed in \cite[Pages 273--274]{MR2902693} 
              which can be extended to the case when $z_a$ and $z_b$ are bounded functions.
              
    \end{enumerate}    
 Now the estimate \eqref{eq:quad_grow} follows from the above properties by applying \cite[Theorem~2.3 and Theorem~2.7]{MR2902693}. The proof is finished.
\end{proof}

\begin{remark}\label{rem:int_oc1}
{\rm  We mention that all the results in this section, except the $H^{2s}$-elliptic regularity result in Lemma \ref{adjoint:exist}, hold for the map $\widetilde S$ by replacing in all the statements and proofs, $(L_D)^s$ by $(-\Delta)_D^s$. For $(-\Delta)_D^s$ only local maximal elliptic regularity can be achieved. More precisely, for the integral fractional Laplacian, we have the following situation concerning the solution $\phi$ of the corresponding adjoint equation.
\begin{itemize}
\item By \cite{Grub} if $0<s<\frac 12$ and $\Omega$ is smooth, then $\phi\in H^{2s}(\Omega)$.
\item If $\frac 12\le s<1$, an example has been given in \cite{Ro-Sj} where $\Omega$ is smooth but $\phi\not\in H^{2s}(\Omega)$.
\end{itemize}
In general, that is, for all $0<s<1$ and an aribitrary bounded open set, one can only achieve a local elliptic maximal regularity, that is, $\phi$ always belongs to $H_{\rm loc}^{2s}(\Omega)$ (see e.g. \cite{BWZ} for more details).}
\end{remark}

\noindent
{\bf Acknowledgement}:
The authors thank both reviewers for their careful reading  of the 
manuscript and for their useful comments that helped to improve the quality of the final version of the paper.

\bibliographystyle{plain}

\bibliography{lit}

\end{document}